\documentclass{amsart}

\usepackage{graphicx}
\usepackage{bm}
\usepackage{color}
\usepackage{natbib}
\usepackage{mathdots}
\usepackage{multicol}
\usepackage{amssymb, amsmath, amsthm}
\usepackage{bm}
\theoremstyle{plain}
\usepackage{graphicx}
\usepackage[aboveskip=-5pt,position=bottom]{caption}
\usepackage{url}
\usepackage{hyperref}

\let\originalleft\left \renewcommand\left{\mathopen{}\mathclose\bgroup\originalleft}
\let\originalright\right \renewcommand\right{\aftergroup\egroup\originalright}

\newtheorem{lemma}{Lemma}[section]

\newtheorem{theorem}[lemma]{Theorem}

\newtheorem{exam}[lemma]{\normalfont \scshape
 Example}
\newtheorem{rem}[lemma]{\normalfont \scshape Remark}

\newcommand{\T}{{\mathpalette\raiseT\intercal}}
\newcommand{\raiseT}[2]{\mspace{-1mu}\raisebox{0.25ex}{$#1#2$}}

\newcommand{\R}{\mathbb{R}}
\newcommand{\N}{\mathbb{N}}
\newcommand{\norm}[1]{\left\Vert#1\right\Vert}
\newcommand{\abs}[1]{\left\vert#1\right\vert}
\newcommand{\set}[1]{\left\{#1\right\}}

\newcommand{\eps}{\varepsilon}

\newcommand{\bfx}{\bm{x}}
\newcommand{\bfzero}{\bm{0}}
\newcommand{\bfinfty}{\bm{\infty}}

\newcommand{\bfone}{\bm{1}}

\newcommand{\bfE}{\bm{E}}
\newcommand{\bfI}{\bm{I}}
\newcommand{\bfM}{\bm{M}}
\newcommand{\bfP}{\bm{P}}

\newcommand{\bfr}{\bm{r}}

\newcommand{\bfs}{\bm{s}}

\newcommand{\bft}{\bm{t}}
\newcommand{\bfU}{\bm{U}}
\newcommand{\bfu}{\bm{u}}
\newcommand{\bfV}{\bm{V}}

\newcommand{\bfX}{\bm{X}}
\newcommand{\bfY}{\bm{Y}}
\newcommand{\bfy}{\bm{y}}
\newcommand{\bfZ}{\bm{Z}}
\newcommand{\bfz}{\bm{z}}

\newcommand{\bfSigma}{\bm{\Sigma}}

\newcommand{\bfeta}{\bm{\eta}}
\newcommand{\bfgamma}{\bm{\gamma}}

\newcommand{\bfxi}{\bm{\xi}}

\newcommand{\neins}{\left\langle n\left(1-\frac{c}{j}\right)\right\rangle}
\newcommand{\neinsn}{\left\langle n\left(1-\frac{c_n}{j}\right)\right\rangle}

\allowdisplaybreaks[4]
\begin{document}

\title[Generalized Pareto Copula]{Testing for a $\delta$-Neighborhood of a generalized Pareto Copula}%
\author{Stefan Aulbach and Michael Falk}
\address{$^1$University of W\"{u}rzburg,
Institute of Mathematics,  Emil-Fischer-Str. 30, 97074 W\"{u}rzburg.}
\email{stefan.aulbach@uni-wuerzburg.de,\newline michael.falk@uni-wuerzburg.de}

\thanks{The first author was supported by DFG grant FA 262/4-2}%


\begin{abstract}
A multivariate distribution function $F$ is in the max-domain of attraction
of an extreme value distribution if and only if this is true for the copula
corresponding to $F$ and its univariate margins. \citet{aulbf11} have shown
that a copula satisfies the extreme value condition if and only if the
copula is tail equivalent to a generalized Pareto copula (GPC). In this
paper we propose a $\chi^2$-goodness-of-fit test in arbitrary dimension for
testing whether a copula is in a certain neighborhood of a GPC. The test
can be applied to stochastic processes as well to check whether the
corresponding copula process is close to a generalized Pareto process.
Since the $p$-value of the proposed test is highly sensitive to a proper
selection of a certain threshold, we also present a graphical tool that
makes the decision, whether or not to reject the hypothesis, more
comfortable.
\end{abstract}

\maketitle

\section{Introduction}

Consider a random vector (rv) $\bfU=(U_1,\dots,U_d)^\T$ whose distribution
function (df) is a copula $C$, i.e., each $U_i$ follows the uniform
distribution on $(0,1)$. The copula $C$ is said to be in  the max-domain of attraction of an extreme value
df (EVD) $G$ on $\R^d$, denoted by $C\in\mathcal D(G)$, if
\begin{equation}\label{def:domain_of_attraction}
C^n\left(\bfone + \frac 1n \bfx\right)\to_{n\to\infty}G(\bfx),\qquad \bfx\le\bfzero\in\R^d.
\end{equation}
The characteristic property of the df $G$ is its \emph{max-stability}, precisely,
\[
G^n\left(\frac{\bfx}n\right) = G(\bfx),\qquad \bfx\in\R^d,\,n\in\N;
\]
see, e.g., \citet[Section 4]{fahure10}.

Let $\bfU^{(1)},\dots, \bfU^{(n)}$ be independent copies of $\bfU$. Equation
\eqref{def:domain_of_attraction} is equivalent with
\[
P\left(n\left(\max_{1\le i\le n}\bfU^{(i)}-\bfone\right)\le \bfx\right) \to_{n\to\infty}G(\bfx),\qquad \bfx\le\bfzero,
\]
where $\bfzero:=(0,\dots,0)^\T\in\R^d$ and $\bfone:=(1,\dots,1)^\T\in\R^d$.
All operations on vectors such as $\max_{1\le i\le n}\bfU^{(i)}$ are meant
componentwise.

From \citet[Corollary 2.2]{aulbf11} we know that $C\in\mathcal D(G)$ if and
only if (iff) there exists a norm $\norm\cdot$ on $\R^d$ such that the copula $C$
satisfies the expansion
\begin{equation}\label{eqn:expansion_of_copula}
C(\bfu)=1-\norm{\bfu-\bfone} + o\left(\norm{\bfu-\bfone}\right)
\end{equation}
uniformly for $\bfu\in[0,1]^d$ as $\bfu\uparrow\bfone$, i.e.,
\begin{equation*}
  \lim_{t\downarrow0} \sup_{\substack{\bfu\in[0,1]^d\setminus\set{\bfone}\\ \norm{\bfu-\bfone}<t}} \frac{\abs{C(\bfu)-\bigl(1-\norm{\bfu-\bfone}\bigr)}}{\norm{\bfu-\bfone}} = 0.
\end{equation*}
In this case the norm $\norm\cdot$ is called a \emph{$D$-norm} and is
commonly denoted by $\norm\cdot_D$, where the additional character $D$ means dependence. The corresponding EVD $G$ is then given
by
\[
G(\bfx)=\exp\left(-\norm{\bfx}_D\right),\qquad \bfx\le\bfzero\in\R^d;
\]
we refer to \citet[Section 5.2]{fahure10} for further details. We have in particular independence of the margins of $G$ iff the $D$-norm $\norm\cdot_D$ is the usual $L_1$-norm $\norm\cdot_1$, and we have complete dependence iff $\norm\cdot_D$ is the maximum-norm $\norm\cdot_\infty$.

If the copula $C$ satisfies
\begin{equation*} 
  C(\bfu) = 1 - \norm{\bfu-\bfone}_D, \qquad \bfu_0\le\bfu\le\bfone,
\end{equation*}
for some $\bfu_0 \in [0,1)^d$, then we refer to it as a \emph{generalized
Pareto copula} (GPC). The characteristic property of a GPC is its
\emph{excursion stability}: The rv $\bfU=(U_1,\dots,U_d)^\T$ follows a GPC
iff there exists $\bfu_0 \in [0,1)^d$ such that
\begin{equation*}\label{eqn:excursion_stability}
P(U_k - 1 > t(u_k - 1),\,k\in K)=t P(U_k>u_k,\,k\in K),\qquad t\in [0,1],
\end{equation*}
for all $\bfu\ge\bfu_0$ and each nonempty subset $K$ of $\set{1,\dots,d}$,
see \citet[Proposition 5.3.4]{fahure10}. Based on this characterization,
\citet{falm09} investigated a test whether $\bfU$ follows a GPC; see also
\citet[Section 5.8]{fahure10}.

It is by no means obvious to find a copula $C$, which does \emph{not} satisfy
$C\in\mathcal D(G)$. An example is given in \citet{kortalb09}. A family of
copulas, which are not in the max-domain of attraction of an EVD but come
arbitrary close to a GPC, is given in Section
\ref{sec:copula_not_in_domain_of_attraction}. Parts of the simulations in
Section \ref{sec:simulations} are based on this family.

If the remainder term in equation \eqref{eqn:expansion_of_copula} satisfies
\begin{equation*}
r(\bfu):=C(\bfu)-\bigl(1-\norm{\bfu-\bfone}_D\bigr)=O\bigl(\norm{\bfu-\bfone}_D^{1+\delta}\bigr)
\tag{$\delta$-n}\label{eqn:def_delta-neighborhood}
\end{equation*}
as $\bfu\uparrow \bfone$ for some $\delta>0$, then the copula $C$ is said to
be in the \emph{$\delta$-neighborhood} of a GPC. Note that
\eqref{eqn:expansion_of_copula} is already implied by
\eqref{eqn:def_delta-neighborhood} and that $O\bigl(\norm{\bfu-\bfone}_D^{1+\delta}\bigr)=O\bigl(\norm{\bfu-\bfone}^{1+\delta}\bigr)$ for an arbitrary norm $\norm\cdot$ on $\R^d$. The significance of such
$\delta$-neighborhoods is outlined in Section \ref{sec:delta-neighborhoods}
where we also give some prominent examples. In Section
\ref{sec:test_in_case_of_copula} we will derive a $\chi^2$-goodness-of-fit
test based on $\bfU^{(1)},\dots,\bfU^{(n)}$ which checks whether the
pertaining copula $C$ satisfies condition \eqref{eqn:def_delta-neighborhood}.

Let $\bfX=(X_1,\dots,X_d)^\T$ be a rv with arbitrary df $F$. It is well-known
\citep{deheu78,deheu83,gal87} that $F$ is in the max-domain of attraction of
an EVD iff this is true for the univariate margins of $F$ together with the
condition that the copula $C_F$ corresponding to $F$ satisfies
\eqref{def:domain_of_attraction}. While there are various tests which check
for the univariate extreme value condition --- see, e.g. \citet{diehh02},
\citet{drehl06} as well as \citet[Section 5.3]{reist07} --- much less has
been done for the multivariate case. Utilizing the \emph{empirical copula},
we can modify the test statistic from Section
\ref{sec:test_in_case_of_copula} and check, whether $C_F$ satisfies condition
\eqref{eqn:def_delta-neighborhood}, based on independent copies
$\bfX^{(1)},\dots,\bfX^{(d)}$. This is the content of Section
\ref{sec:arbitrary_random_vector}.

Sections \ref{sec:test_GPCP} and \ref{sec:test_processes_grid} carry the
results of Section \ref{sec:test_multivariate} over to function space. The
aim is to test whether the copula process of a given stochastic process in
$C[0,1]$ is in a $\delta$-neighborhood of a generalized Pareto copula process
(GPCP). While Section \ref{sec:test_GPCP} deals with copula processes or
general processes as a whole, respectively, Section
\ref{sec:test_processes_grid} considers the case that the underlying
processes are observed at a finite grid of points only.

In order to demonstrate the performance of our test, Section
\ref{sec:simulations} states the results of a simulation study. Since the
results from the previous sections highly depend on a proper choice of some
threshold, we also present a graphical tool that makes the decision, whether
or not to reject the hypothesis, more comfortable.

\section{A copula not in max-domain of attraction}\label{sec:copula_not_in_domain_of_attraction}

The following result provides a one parametric family of bivariate rv, which
are easy to simulate. Each member of this family has the property that its
corresponding copula does not satisfy the extreme value condition
\eqref{eqn:expansion_of_copula}. However, as the parameter tends to zero, the
copulas of interest come arbitrarily close to a GPC, which, in general, is in the
domain of attraction of an EVD.

\begin{lemma} \label{lem:copula_not_in_domain_of_attraction}
Let the rv $V$ have df
\[
H_\lambda(u):= u \bigl(1+\lambda\sin(\log(u))\bigr),\qquad 0\le u\le 1,
\]
where $\lambda\in\bigl[-\frac{\sqrt2}2,\frac{\sqrt2}2\bigr]$. Note that
$H_\lambda(0)=0$, $H_\lambda(1)=1$ and $H_\lambda'(u)\ge 0$ for $0<u<1$.
Furthermore let the rv $U$ be independent of $V$ and uniformly distributed on
$(0,1)$. Put $S_1:=U=:1-S_2$. Then the copula $C_\lambda$ corresponding to
the bivariate rv
\begin{equation} \label{eqn:copula_not_in_domain_of_attraction}
  \bfX:= -\frac V2\left(\frac 1{S_1},\frac 1{S_2}\right)^\T\in (-\infty,0]^2
\end{equation}
is not in the domain of attraction of a multivariate EVD if $\lambda\ne0$,
whereas $C_0$ is a GPC with corresponding $D$-norm
\begin{equation*}
  \norm{\bfx}_D = \norm{\bfx}_1 - \frac{\abs{x_1}\abs{x_2}}{\norm{\bfx}_1}
\end{equation*}
for $\bfx=(x_1,x_2)^\T\ne\bfzero$.
\end{lemma}

Denote by $F_\lambda$ the df of $-V/S_1=_D-V/S_2$. Elementary computations
yield that it is given by
\begin{equation*} 
    F_\lambda(x) =
    \begin{cases}
      \abs{x}^{-1}\left(\frac12 + \frac\lambda5\right), &\text{if }x\le-1,\\
      1-\abs{x}\Bigl(\frac12+\frac\lambda5\big(2\sin(\log\abs{x}) -\cos(\log\abs{x})\big)\Bigr), &\text{if } {-1}<x<0,
    \end{cases}
  \end{equation*}
and, thus, $F_\lambda$ is continuous and strictly increasing on
$(-\infty,0]$.

\begin{proof}[Proof of Lemma \ref{lem:copula_not_in_domain_of_attraction}]
We show that
\[
\lim_{s\downarrow 0} \frac{1-C_\lambda(1-s,1-s)} s
\]
does not exist for
$\lambda\in\bigl[-\frac{\sqrt2}2,\frac{\sqrt2}2\bigr]\setminus\set0$. Since
$C_\lambda$ coincides with the copula of $2\bfX$ we obtain setting
$s=1-F_\lambda(t)$, $t\uparrow 0$,
\begin{align*}
  \frac{1-C_\lambda\bigl(F_\lambda(t),F_\lambda(t)\bigr)}{1-F_\lambda(t)}
  &= \frac{1-P\bigl(-V/S_1\le t, -V/S_2\le t\bigr)}{1-P\bigl(-V/S_1\le t\bigr)} \\
  &= \frac{1-P\bigl(V\ge \abs{t}\max\{U,1-U\}\bigr)}{1-P\bigl(V\ge \abs{t}U)} \\
  &= \frac{\int_0^1 P\bigl(V\le \abs t\max\{u,1-u\}\bigr)\, du}{\int_0^1 P\bigl(V\le \abs t u\bigr)\, du}\\
  &= \frac{\int_0^{1/2} H_\lambda\bigl(\abs t(1-u)\bigr)\, du + \int_{1/2}^1 H_\lambda\bigl(\abs t u\bigr)\, du}{\int_0^1 H_\lambda\bigl(\abs t u\bigr)\, du}\\
  &=2 \frac{\int_{1/2}^1 H_\lambda\bigl(\abs t u\bigr)\, du}{\int_0^1 H_\lambda\bigl(\abs t u\bigr)\, du}.
\end{align*}
The  substitution  $u\mapsto u/\abs t$ yields
\begin{equation*}
  1 - \frac12\frac{1-C_\lambda\bigl(F_\lambda(t),F_\lambda(t)\bigr)}{1-F_\lambda(t)}
  = 1 - \frac{\int_{\abs t/2}^{\abs t} H_\lambda(u)\, du}{\int_0^{\abs t} H_\lambda(u)\, du}
  = \frac{\int_0^{\abs t/2} H_\lambda(u)\, du}{\int_0^{\abs t} H_\lambda(u)\, du}
\end{equation*}
where we have for each $0<c\le1$
\begin{align*}
  \int_0^c H_\lambda(u)\, du
  = \frac{c^2}{2} + \lambda \int_0^c  u\sin(\log(u))\, du.
\end{align*}
and
\begin{equation*} 
  \int_0^c  u^2\cdot \frac1u\sin(\log(u))\, du
  = \frac{c^2}{5} \bigl(2\sin(\log(c)) - \cos(\log(c))\bigr)
\end{equation*}
which can be seen by applying integration by parts twice. Hence we
obtain
\begin{equation*}
  \frac{\int_0^{\abs t/2} H_\lambda(u)\, du}{\int_0^{\abs t} H_\lambda(u)\, du}
  = \frac14 \frac{\frac12 + \frac{\lambda}{5} \bigl(2\sin(\log\abs t - \log(2)) - \cos(\log\abs t - \log(2))\bigr)}{\frac12 + \frac{\lambda}{5} \bigl(2\sin(\log\abs t) - \cos(\log\abs t)\bigr)},
\end{equation*}
whose limit does not exist for $t\uparrow0$ if
$\lambda\in\bigl[-\frac{\sqrt2}2,\frac{\sqrt2}2\bigr]\setminus\set0$;
consider, e.\,g., the sequences $t_n^{(1)}=-\exp\bigl((1-2n)\pi\bigr)$ and
$t_n^{(2)}=-\exp\bigl((1/2-2n)\pi\bigr)$ as $n\to\infty$.

On the other hand, elementary computations show for
$\bfx=(x_1,x_2)^\T\in(-\infty,0]^2\setminus\set{\bfzero}$
\begin{equation*}
  \lim_{\eps\downarrow0}\frac{1-C_0(1+\eps\bfx)}{\eps}
  = 2 E(\max\set{\abs{x_1}S_1,\abs{x_2}S_2})
  = \norm{\bfx}_1 - \frac{\abs{x_1}\abs{x_2}}{\norm{\bfx}_1}.
\end{equation*}
The remaining assertion is thus implied by Section 2 of \citet{aulbf11}.
\end{proof}

\begin{rem}
Similar results as in Lemma \ref{lem:copula_not_in_domain_of_attraction} can
be obtained for different distributions of the rv $(S_1,S_2)$ in
\eqref{eqn:copula_not_in_domain_of_attraction}, which is still assumed to be
independent of $V$. If $\lambda=0$, then $S_1=S_2=U$ implies
$\norm{{\cdot}}_D = \norm{{\cdot}}_\infty$, whereas $S_1=U_1$, $S_2=U_2$
gives
\begin{equation*}
  \norm{\bfx}_D = \norm{\bfx}_\infty + \frac{\left(\norm{\bfx}_1-\norm{\bfx}_\infty\right)^2}{3\norm{\bfx}_\infty},\qquad \bfx\ne\bfzero,
\end{equation*}
where $U,U_1,U_2$ are independent and uniformly distributed on $[0,1]$.
However, if $\lambda\ne0$, then we obtain
\begin{equation*}
  \frac{1-C_\lambda\bigl(F_\lambda(t),F_\lambda(2t)\bigr)}{1-F_\lambda(t)}
  = 2 \frac{5+2\lambda \bigl(\sin\bigl(\log(2\abs{t})\bigr)-\cos\bigl(\log(2\abs{t})\bigr)\bigr)}{5+2\lambda\bigl(2\sin(\log\abs{t})-\cos(\log\abs{t})\bigr)},
  \quad t\in\Bigl(-\frac12,0\Bigr),
\end{equation*}
and
\begin{equation*}
  \frac{1-C_\lambda\bigl(F_\lambda(t),F_\lambda(t)\bigr)}{1-F_\lambda(t)}
  = 1 + \frac{5+6\lambda \sin(\log\abs{t})}{15+6\lambda\bigl(2\sin(\log\abs{t})-\cos(\log\abs{t})\bigr)},
  \quad t\in(-1,0),
\end{equation*}
respectively. Note that both terms have no limit for $t\uparrow0$; consider,
e.\,g., the sequences $\bigl(t_n^{(1)}\bigr)_n$ and $\bigl(t_n^{(2)}\bigr)_n$
as in the proof of Lemma \ref{lem:copula_not_in_domain_of_attraction}.
\end{rem}

\section{$\delta$-neigborhoods} \label{sec:delta-neighborhoods}
The significance of the $\delta$-neighborhood of a GPC can be seen as
follows. Denote by $R:=\set{\bft\in[0,1]^d:\,\norm{\bft}_1=\sum_{i=1}^d
t_i=1}$ the unit sphere in $[0,\infty)^d$ with respect to the $L_1$-norm
$\norm\cdot_1$. Take an arbitrary copula $C$ on $\R^d$ and put for $\bft\in
R$
\[
C_{\bft}(s):=C(\bfone+s\bft),\qquad s\le 0.
\]
Then $C_{\bft}$ is a univariate df on $(-\infty,0]$ and the copula $C$ is
obviously determined by the family
\[
\mathcal P(C):=\set{C_{\bft}:\,\bft\in R}
\]
of univariate \emph{spectral df} $C_{\bft}$. This family $\mathcal P(C)$ is
the \emph{spectral decomposition} of $C$; cf. \citet[Section 5.4]{fahure10}.
A copula $C$ is, consequently, in $\mathcal D(G)$ with corresponding $D$-norm
$\norm{\cdot}_D$ iff its spectral decomposition satisfies
\[
C_{\bft}(s)=1+s\norm{\bft}_D+o(s),\qquad \bft\in R,
\]
as $s\uparrow 0$. The copula $C$ is in the $\delta$-neighborhood of the GPC
$C_D$ with $D$-norm $\norm\cdot_D$ iff
\begin{equation}\label{eqn:characterization_of_delta_neighborhood_via_sdf}
1-C_{\bft}(s)=\left(1-C_{D,\bft}(s)\right)\left(1+O\left(\abs s^\delta\right)\right)
\end{equation}
uniformly for $\bft\in R$ as $s\uparrow 0$. In this case we know from
\citet[Theorem~5.5.5]{fahure10} that
\begin{equation}\label{eqn:rate_of_convergence}
\sup_{\bfx\in(-\infty,0]^d}\abs{C^n\left(\bfone + \frac 1n\bfx\right) - \exp\left(-\norm{\bfx}_D\right)} = O\left(n^{-\delta}\right).
\end{equation}
Under additional differentiability conditions on $C_{\bft}(s)$ with respect
to $s$, also the reverse implication \eqref{eqn:rate_of_convergence}
$\implies$ \eqref{eqn:characterization_of_delta_neighborhood_via_sdf} holds;
cf. \citet[Theorem 5.5.5]{fahure10}. Thus the $\delta$-neighborhood of a GPC,
roughly, collects those copula with a polynomial rate of convergence of
maxima.

\begin{exam}[Archimedean Copula]\upshape
Let
\[
C(\bfu)=\varphi^{[-1]}\left(\sum_{k=1}^d\varphi(u_i)\right),\qquad \bfu=(u_1,\dots,u_d)^\T\in[0,1]^d,
\]
be an Archimedean copula with generator function
$\varphi:\,[0,1]\to[0,\infty]$ and $\varphi^{[-1]}(t) := \inf\{u\in[0,1]:
\varphi(u)\le t\}$, $0\le t\le\infty$. The function $\varphi$ is in
particular strictly decreasing, continuous and satisfies $\varphi(1)=0$; for
a complete characterization of the function $\varphi$ we refer to
\citet{mcnn09}.

Suppose that $\varphi$ is differentiable on $[\eps,1]$ for some $\eps<1$ with
derivative satisfying
\begin{equation}\label{cond:on_generator_function}
\varphi'(1)<0,\quad \varphi'(1-h)=\varphi'(1)+O(h^\delta)
\end{equation}
for some $\delta>0$ as $h\downarrow 0$. Then $C$ is in the
$\delta$-neighborhood of a GPC with $D$-norm given by
$\norm{\bfx}_D=\norm{\bfx}_1$, $\bfx=(x_1,\dots,x_d)^\T\in\R^d$.

The Clayton family with
$\varphi_\vartheta(t)=\vartheta^{-1}\left(t^{-\vartheta}-1\right)$,
$\vartheta\in[-1,\infty)\backslash\set 0$ satisfies condition
\eqref{cond:on_generator_function} with $\delta=1$ if $\vartheta>-1$. The
parameter $\vartheta=-1$ yields a GPC with corresponding $D$-norm
$\norm\cdot_1$.

The Gumbel-Hougard family with $\varphi_\vartheta(t)=(-\log(t))^\vartheta$,
$\vartheta\in[1,\infty)$ does not satisfy condition
\eqref{cond:on_generator_function}. But for $\vartheta\in[1,2)$ it is in the
$\delta$-neighborhood with $\delta=2-\vartheta$ of a GPC having $D$-norm
$\norm{\bfx}_\vartheta=\left(\sum_{i=1}^d\abs{x_i}^\vartheta\right)^{1/\vartheta}$.
For general results on the limiting distributions of Archimedean copulas we
refer to \citet{chas09} and \citet{larn11}.
\end{exam}

\begin{exam}[Normal Copula]\upshape
Let $C$ be a normal copula, i.e., $C$ is the df of
$\bfU=((\Phi(X_1),\dots,\Phi(X_d))^\T$, where $\Phi$ denotes the standard
normal df and $(X_1,\dots,X_d)^\T$ follows a multivariate normal distribution
$N(\bm{\mu},\bfSigma)$ with mean vector $\bm\mu=\bfzero$ and covariance
matrix $\bfSigma=(\rho_{ij})_{1\le i,\le j}$, where $\rho_{ii}=1$, $1\le i\le
d$. If $-1<\rho_{ij}<0$ for $1\le i\not= j\le d$, then $C$ is in the
$\delta$-neighborhood of a GPC $C_D$ with $\norm\cdot_D=\norm\cdot_1$ and
\[
\delta=\min_{1\le i\not=j\le d}\frac{\rho_{ij}^2}{1-\rho_{ij}^2}.
\]

\begin{proof}
Put $\bfy:=\left(\Phi(x_i)\right)_{i=1}^d$. Then we have
\begin{align*}
C(\bfy)&= P\left(\Phi(X_i)\le y_i,\,1\le i\le d\right)\\
&=P(X_i\le x_i,\,1\le i\le d)\\
&=1- P\left(\bigcup_{i=1}^d\set{X_i>x_i}\right)\\
&=1 - \sum_{i=1}^d P(X_i>x_i)+ \sum_{T\subset\set{1,\dots,d},\abs T\ge 2} (-1)^{\abs T} P(X_i>x_i,\,i\in T)\\
&=1-\norm{\bfone-\bfy}_1+ \sum_{T\subset\set{1,\dots,d},\abs T\ge 2} (-1)^{\abs T} P(X_i>x_i,\,i\in T)
\end{align*}
by the inclusion-exclusion theorem.

By $c$ we denote in what follows a positive generic constant. We have
\begin{align*}
\abs{\sum_{T\subset\set{1,\dots,d},\abs T\ge 2} (-1)^{\abs T} P(X_i>x_i,\,i\in T)}
\le c \sum_{1\le i\not= j\le d} P(X_i>x_i,X_j>x_j).
\end{align*}
We will show that for all $i\not=j$
\begin{equation}\label{eqn:first_bound_in_normal_copula_example}
\frac{P(X_i>x_i,X_j>x_j)}{\left(\sum_{m=1}^d(1-\Phi(x_m))\right)^{1+\delta}} \le c,\qquad 1\le i\not=j\le d,
\end{equation}
for $\bfx\ge\bfx_0$, where $\bfx_0\in\R^d$ is specified later. This,
obviously, implies the assertion.

Equation \eqref{eqn:first_bound_in_normal_copula_example} is implied by the
inequality
\begin{equation}\label{eqn:second_bound_in_normal_copula_example}
\frac{P(X_i>x_i,X_j>x_j)^{\frac 1{1+\delta}}}{1-\Phi(x_i)+ 1-\Phi(x_j)}\le c,\qquad 1\le i\not=j\le d,
\end{equation}
for $\bfx\ge\bfx_0$, which we will establish in the sequel.

Fix $i\not=j$. To ease the notation we put
$X:=X_i,\,Y:=X_j,\,x:=x_i,\,y:=x_i,\,\rho:=\rho_{ij}$. The covariance matrix
of $(X,Y)^\T$ is $\bfSigma_{X,Y}=\begin{pmatrix} 1 &\rho\\ \rho & 1
\end{pmatrix}$, its inverse is $\bfSigma_{X,Y}^{-1} =\frac 1 {1-\rho^2}
\begin{pmatrix} 1 &-\rho\\ -\rho & 1 \end{pmatrix}$ and, hence,
\[
\bfSigma_{X,Y}^{-1} \begin{pmatrix} x\\ y\end{pmatrix} =\frac 1{1-\rho^2}  \begin{pmatrix} x-\rho y\\ y-\rho x\end{pmatrix} >\bfzero
\]
if $x,y>0$; recall that $\rho<0$. From \citet{sav62} (see also \citet{ton89}
and \citet{hash03}) we obtain the bound
\begin{equation}\label{eqn:third_equation_in_normal_copula_example}
P(X>x, Y> y)\le c \frac 1{(x-\rho y)(y-\rho x)} \exp\left(-\frac{x^2-2\rho xy+y^2}{2(1-\rho^2)}\right),\quad x,y>0.
\end{equation}
By the obvious inequality
\[
\delta=\min_{1\le k\not=m\le d}\frac{\rho_{km}^2}{1-\rho_{km}^2}\le \frac{\rho^2}{1-\rho^2}
\]
we obtain
\[
\frac 1{1+\delta}\ge 1-\rho^2
\]
and, thus, equation \eqref{eqn:third_equation_in_normal_copula_example}
implies
\[
P(X>x,Y>y)^{\frac 1{1+\delta}} \le c\frac 1{((x-\rho y)(y-\rho x))^{1-\rho^2}} \exp\left(-\frac{x^2-2\rho xy+y^2}{2(1-\rho^2)}\right).
\]

From the fact that $1-\Phi(x)\sim \varphi(x)/x$ as $x\to\infty$, where
$\varphi=\Phi'$ denotes the standard normal density, we obtain for $x,y\ge
x_0$
\begin{align*}
\frac{P(X>x,Y>y)^{\frac 1{1+\delta}}} {1-\Phi(x)+1-\Phi(y)} &\le c \frac{x\exp\left(\frac{x^2} 2\right)+ y\exp\left(\frac{y^2} 2\right)} {((x-\rho y)(y-\rho x))^{1-\rho^2} \exp\left(\frac{x^2-2\rho xy+y^2}2\right)}\\
&= c \frac{x\exp\left(-\frac{y^2} 2\right)+ y\exp\left(-\frac{x^2} 2\right)} {((x-\rho y)(y-\rho x))^{1-\rho^2} \exp(-\rho xy)}\\
&\le c \frac{x\exp\left(-\frac{y^2} 2\right)+ y\exp\left(-\frac{x^2} 2\right)} {(xy)^{1-\rho^2} \exp(-\rho xy)}\\
&\le c \frac{(xy)^{\rho^2}}{\exp(-\rho xy)}\\
&\le c;
\end{align*}
recall that $\rho<0$. This implies equation
\eqref{eqn:second_bound_in_normal_copula_example} and, thus, the assertion.
\end{proof}
\end{exam}

\section{A test based on copula data} \label{sec:test_multivariate}
This section deals with deriving a test for condition
\eqref{eqn:def_delta-neighborhood} based on independent copies
$\bfU^{(1)},\dots, \bfU^{(n)}$ of the rv $\bfU =
\left(U_1,\dots,U_d\right)^\T$ having df $C$. Put for $s<0$
\begin{equation*}
  S_{\bfU}(s):=\sum_{i=1}^d 1_{(s,\infty)}(U_i-1),
\end{equation*}
which is a discrete version of the sojourn time that $\bfU$ spends above the
threshold $1+s$; see \citet{falkho11} for details as well as Section
\ref{sec:test_GPCP}. If $C$ satisfies condition
\eqref{eqn:def_delta-neighborhood}, then we obtain with
$\bfs:=s\bfone\in\R^d$
\begin{align}
  P\left(S_{\bfU}(s)=0\right)&=P(\bfU\le \bfone+\bfs)\nonumber\\
  &=C(\bfone+\bfs)\nonumber\\
  &=1-\abs s m_D + O\left(\abs s^{1+\delta}\right). \label{eqn:expansion_of_binomial_probability}
\end{align}
The constant $m_D:=\norm{\bfone}_D$, which is always between $1$ and $d$,
measures the tail dependence of the margins of $C$. It is the extremal
coefficient \citep{smith90} and equal to one in case of complete dependence
of the margins and equal to $d$ in case of independence \citep{taka88}; we
refer to \citet[Section 4.4]{fahure10} for further details.

\subsection{Observing copula data directly} \label{sec:test_in_case_of_copula}

In order to test for condition \eqref{eqn:expansion_of_binomial_probability},
we fit a grid in the upper tail of the copula $C$ and observe the exceedances
with respect to this grid: Choose $k\in\N$ and put for $0<c<1$
\begin{equation}\label{eqn:definition_of_n_j(c)}
  n_j(c):=\sum_{i=1}^n 1_{(0,\infty)}\Bigl(S_{\bfU^{(i)}}\Bigl(-\frac cj\Bigr)\Bigr),\qquad 1\le j\le k.
\end{equation}
$n_j(c)$ is the number of those rv $\bfU^{(i)}$ among the independent copies
$\bfU^{(1)},\dots,\bfU^{(n)}$ of $\bfU$, whose sojourn times above $1 -
\frac{c}{j}$ are positive, i.e., at least one component of $\bfU^{(i)}$
exceeds the threshold $1 - \frac{c}{j}$. On the other hand,
\begin{equation*}
  n-n_j(c)
  = \sum_{i=1}^n 1_{\set 0}\Bigl(S_{\bfU^{(i)}}\Bigl(-\frac cj\Bigr)\Bigr)
  = \sum_{i=1}^n 1_{\left[\bfzero,\left(1-\frac cj\right)\bfone\right]}\bigl(\bfU^{(i)}\bigr)
\end{equation*}
is the number of those rv $\bfU^{(i)}$ whose realizations are below the
vector with constant entry $1 - \frac{c}{j}$.

If $C$ satisfies condition \eqref{eqn:def_delta-neighborhood}, then each
$n_j(c)$ is binomial $B(n,p_j(c))$-distributed with
\begin{equation*}
  p_j(c)
  := 1-P\Bigl(S_{\bfU}\Bigl(-\frac cj\Bigr)=0\Bigr)
  = \frac cj m_D + O\left(c^{1+\delta}\right).
\end{equation*}
Motivated by the usual $\chi^2$-goodness-of-fit test, we consider in what
follows the test statistic
\begin{equation}\label{eqn:definition_of_test_statistic_T_n}
  T_n(c)
  := \frac{\sum_{j=1}^k\left(j\, n_j(c) - \frac 1k\sum_{\ell=1}^k \ell\, n_\ell(c)\right)^2}{\frac 1k\sum_{\ell=1}^k \ell\, n_\ell(c)}
\end{equation}
which does not require the constant $m_D$ to be known. By $\to_D$ we denote
ordinary convergence in distribution as the sample size $n$ tends to
infinity.

\begin{theorem}\label{thm:limit_distribution_of_T_n}
Suppose that $C$ satisfies condition \eqref{eqn:def_delta-neighborhood} with
$\delta>0$. Let $c=c_n$ satisfy $c_n\to 0$, $nc_n\to\infty$ and
$nc_n^{1+2\delta}\to0$ as $n\to\infty$. Then we obtain
\[
T_n(c_n)\to_D  \sum_{i=1}^{k-1} \lambda_i \xi_i^2,
\]
where  $\xi_1,\dots,\xi_{k-1}$ are independent and standard normal distributed rv and
\[
\lambda_i=\frac 1 {4\sin^2\left(\frac i k\frac \pi 2\right)},\qquad 1\le i\le k-1.
\]
\end{theorem}

\begin{rem}\upshape
We have $\lambda_1=1/2$ in case $k=2$, and $\lambda_1=1$, $\lambda_2=1/3$ in
case $k=3$. If the copula $C$ is a GPC, then the condition
$nc_n^{1+2\delta}\to_{n\to\infty}0$ in the preceding result can be dropped.
\end{rem}

The proof of Theorem \ref{thm:limit_distribution_of_T_n} shows that
$\sum_{i=1}^{k-1} \lambda_i \xi_i^2$ equals the distribution of
$\sum_{i=1}^k\left(B(i)- \frac 1k \sum_{j=1}^kB(j)\right)^2$, where
$(B(t))_{t\ge 0}$ is a standard Brownian motion on $[0,\infty)$.  Computing
the expected values, we obtain as a nice by-product the equation
\[
\sum_{i=1}^{k-1} \frac 1 {4\sin^2\left(\frac i k\frac \pi 2\right)} = \frac{(k-1)(k+1)}6,\qquad k\ge 2.
\]
Using characteristic functions it is straightforward to prove that
\[
\frac 6{(k-1)(k+1)}\sum_{i=1}^{k-1}\lambda_i\xi_i^2\to_D \frac{24}{\pi^2}\sum_{i=1}^\infty \frac 1{i^2}\xi_i^2
\]
as $k\to\infty$. Taking expectations on both sides motivates the well-known equality
\[
\sum_{i=1}^\infty \frac 1{i^2} =\frac{\pi^2} 6.
\]

\begin{proof}[Proof of Theorem \ref{thm:limit_distribution_of_T_n}]

Lindeberg's central limit theorem implies
\begin{equation*}
  \frac 1{(nc_n)^{1/2}} \bigl(j\, n_j(c_n) - n c_nm_D\bigr)
  \to_D N\bigl(0,j m_D\bigr)
\end{equation*}
and, thus,
\begin{equation*}
  \frac{j\, n_j(c_n)}{nc_n} \to_{n\to\infty}m_D\quad\text{in probability},\qquad 1\le j\le k,
\end{equation*}
yielding
\begin{equation*}
  \frac 1{nc_nk} \sum_{j=1}^k j\, n_j(c_n)
  \to_{n\to\infty} m_D \quad\mathrm{in\ probability}.
\end{equation*}
We, therefore, can substitute the denominator in the test statistic
$T_n(c_n)$ by $nc_nm_D$, i.e., $T_n(c_n)$ is asymptotically equivalent with
\begin{align*}
&\frac 1 {nc_n m_D} \sum_{j=1}^k\left(j\, n_j(c_n) - \frac 1k\sum_{\ell=1}^k \ell\, n_\ell(c_n)\right)^2\\
&= \frac 1{nc_n m_D}
   \begin{pmatrix}
     1\cdot n_1(c_n) - nc_nm_D \\
     \vdots \\
     k\cdot n_k(c_n) - nc_nm_D
   \end{pmatrix}^\T
   \left(\bfI_k-\frac 1 k \bfE_k\right)
   \begin{pmatrix}
     1\cdot n_1(c_n) - nc_nm_D \\
     \vdots \\
     k\cdot n_k(c_n) - nc_nm_D
   \end{pmatrix}\\
&= \bfY_n^\T \left(\bfI_k-\frac 1 k \bfE_k\right) \bfY_n,
\end{align*}
where $\bfY_n = (Y_{n,1},\dots,Y_{n,k})^\T$ with
\begin{equation*}
  Y_{n,j} = \frac 1{(nc_n m_D)^{1/2}}\bigl(j\, n_j(c_n) - nc_nm_D\bigr),\qquad 1\le j\le k,
\end{equation*}
$\bfI_k$ is the $k\times k$ unit matrix and $\bfE_k$ that $k\times k$-matrix
with constant entry 1. Note that the matrix $\bfP_k:= \bfI_k-k^{-1}\bfE_k$ is
a \emph{projection matrix}, i.e., $\bfP_k=\bfP_k^\T=\bfP_k^2$, and that
$\bfP_k\bfx=\bfzero$ for every vector $\bfx\in\R^k$ with constant entries.

The Cram\'{e}r-Wold theorem and Lindeberg's central limit theorem imply $\bfY_n
\to_D N(\bfzero,\bfSigma)$, where the $k\times k$-covariance matrix
$\bfSigma=(\sigma_{ij})$ is given by
\begin{align*}
  \sigma_{ij}
  &= \lim_{n\to\infty} \frac 1{nc_nm_D} E\left(\bigl(i\, n_i(c_n) - nc_nm_D\bigr)\bigl(j\, n_j(c_n) - nc_nm_D\bigr)\right) \\
  &= \lim_{n\to\infty} \frac{ij}{c_nm_D} E\left[1_{(0,\infty)}\left(S_{\bfU}\left(-\frac{c_n}{i\vphantom{j}}\right)\right)\, 1_{(0,\infty)}\left(S_{\bfU}\left(-\frac{c_n}{j}\right)\right)\right] \\
  &= \lim_{n\to\infty} \frac{ij}{c_nm_D} P\left(S_{\bfU}\left(-\frac{c_n}{\max(i,j)}\right)>0\right) \\
  &= \frac{ij}{\max(i,j)}
   = \min\left(i, j\right).
\end{align*}
Note that $\bfSigma= \bfM_k \bfM_k^{\T}$ where the $k\times k$-matrix
$\bfM_k$ is defined by
\begin{equation*}
\bfM_k:= \bigl(1_{[j,\infty)}(i)\bigr)_{1\le i,j\le k}=
\begin{pmatrix}
1&0&0&\dots&0\\
1&1&0&\dots&0\\
\vdots&\vdots&\ddots&\ddots&\vdots\\
1&1&\dots&1&0\\
1&1&\dots&1&1
\end{pmatrix}.
\end{equation*}
Altogether we obtain
\begin{equation*}
  T_n(c_n) \to_D \bfxi^\T \bfM_k^\T\left(I_k-\frac 1k E_k\right)\bfM_k \bfxi
\end{equation*}
with a $k$-dimensional standard normal rv $\bfxi = (\xi_1,\dots,\xi_k)^\T$.
It is well-known that the eigenvalues of
\begin{equation*}
  k\,\bfM_k^\T\left(\bfI_k-\frac 1k \bfE_k\right)\bfM_k
  = \bigl(k\min(i-1,j-1) - (i-1)(j-1)\bigr)_{1\le i,j\le k}
\end{equation*}
are
\begin{equation*}
  k\lambda_j = \frac k{4 \sin^2\left(\frac jk\frac\pi 2\right)},\ j=1,\dots,k-1,\quad\text{and}\quad \lambda_k=0,
\end{equation*}
see, for example, \cite{andst97} or \cite{forcu97}, with corresponding
orthonormal eigenvectors
\begin{equation*}
  \bfr_j = \sqrt{\frac2k} \left(\sin\left(\frac{(i-1)j\pi}{k}\right)\right)_{1\le i\le k},\ j=1,\dots,k-1, \quad\text{and}\quad \bfr_k=(1,0,\dots,0)^\T.
\end{equation*}
This implies
\begin{equation*}
  T_n(c_n) \to_D \bfxi^\T \mathinner{\operatorname{diag}(\lambda_1,\dots,\lambda_{k-1},0)} \bfxi
  = \sum_{i=1}^{k-1} \lambda_i \xi_i^2
\end{equation*}
as asserted.
\end{proof}

To evaluate the performance of the above test we consider in what follows $n$ independent copies $\bfU_1,\dots,\bfU_n$ of the rv $\bfU$, whose df $C$ satisfies for some $\delta>0$ the expansion
\[
C(\bfu)=1-\norm{\bfu-\bfone}_D-J\left(\frac{\bfu-\bfone}{\norm{\bfu-\bfone}_1}\right) \norm{\bfu-\bfone}_D^{1+\delta} + o\left(\norm{\bfu-\bfone}_D^{1+\delta}\right)
\]
as $\bfu\uparrow 1$, uniformly for $\bfu\in[0,1]^d$, where $J(\cdot)$ is an arbitrary function on the set $\set{\bfz\le\bfzero\in\R^d:\,\norm{\bfz}_1=1}$ of directions in $(-\infty,0]^d$. The above condition specifies the remainder term in the $\delta$-neighborhood condition \eqref{eqn:def_delta-neighborhood}. We obtain for $c\in(0,1)$ and $j\in\N$
\begin{align*}
p_j(c)&=1-P\left(\bfU\le 1-\frac cj\right)\\
&=\frac cj m_D + K\left(\frac cj m_D\right)^{1+\delta} + o\left(c^{1+\delta}\right),
\end{align*}
where $K:=J(\bfone/d)$. For $n_j(c)$ as defined in \eqref{eqn:definition_of_n_j(c)} we obtain by elementary arguments
\[
\frac 1 {(nc_nm_D)^{1/2}} (jn_j(c_n) - nc_nm_D) \to_D N\left(\frac{s^{1/2}m_D^{1/2+\delta}K}{j^\delta},j\right)
\]
if $c_n\to 0$, $nc_n\to\infty$ and $nc_n^{1+2\delta}\to s\ge 0$ as $n\to\infty$. Repeating the arguments in the proof of Theorem \ref{thm:limit_distribution_of_T_n} then yields for the test statistic $T_n$ defined in \eqref{eqn:definition_of_test_statistic_T_n}
\[
T_n(c_n)\to_D\sum_{i=1}^{k-1}\lambda_i(\xi_i+\mu_i)^2,
\]
where
\[
\mu_i:= K\sqrt{\frac{2s}k}m_D^{1/2+\delta} \sum_{j=1}^{k-1} \frac 1{(j+1)^\delta}\sin\left(j\frac{ i \pi}k\right),\qquad 1\le i\le k-1.
\]
Note that each $\mu_i> 0$ if $Ks>0$. With $K>0$, the fact that $C$ is not a GPC is, therefore, detected at an arbitrary level-one error iff $nc_n^{1+2\delta}\to\infty$. Testing for a $\delta$-neighborhood requires however the rate $nc_n^{1+2\delta}\to 0$.

\subsection{The case of an arbitrary random
vector}\label{sec:arbitrary_random_vector}

Consider a rv $\bfX=(X_1,\dots,X_d)^\T$ whose df $F$ is continuous and the
copula $C_F(\bfu)=F\left(F_1^{-1}(u_1),\dots,F_d^{-1}(u_d)\right)$,
$\bfu\in(0,1)^d$, corresponding to $F$ satisfies condition
\eqref{eqn:expansion_of_copula}. Now we will modify the test statistic
$T_n(c_n)$ from above to obtain a test which checks whether $C_F$ satisfies
condition \eqref{eqn:def_delta-neighborhood} with $F_i$ unknown,
$i=1,\dots,d$. We denote in what follows by
$H^{-1}(q):=\inf\set{t\in\R:\,H(t)\ge q}$, $q\in(0,1)$, the generalized
inverse of an arbitrary univariate df $H$.

Let $\bfX^{(i)}=\bigl(X_1^{(i)},\dots, X_d^{(i)}\bigr)^\T$, $i=1,\dots,n$, be
independent copies of $\bfX$ and fix $k\in\set{2,3,\dots}$. It turns out
that, contrary to \eqref{eqn:definition_of_n_j(c)}, it is too ambitious to
take all $n$ observations into account. So choose an arbitrary subset $M(n)$
of $\set{1,\dots,n}$ of size $\abs{M(n)}=m_n$, and put for $0<c<1$
\begin{align*}
  n_{j,M(n)}(c)
  &:=\sum_{i\in M(n)} 1_{(0,\infty)}\left(\sum_{r=1}^d 1_{\left(F_r^{-1}\left(1-\frac{c}{j}\right), \infty\right)}\bigl(X_r^{(i)}\bigr)\right) \\
  &\hphantom{:}= m_n - \sum_{i\in M(n)}  1_{\left(-\bfinfty, \bfgamma_j(c)\right]}\bigl(\bfX^{(i)}\bigr),
  \qquad 1\le j\le k,
\end{align*}
which is the number of all rv $\bigl(\bfX^{(i)}\bigr)_{i\in M(n)}$ exceeding
the vector
\begin{equation*}
  \bfgamma_j(c)
  := \left(F_1^{-1}\left(1-\frac{c}{j}\right),\dots,F_d^{-1}\left(1-\frac{c}{j}\right)\right)^\T
\end{equation*}
in at least one component. In Section \ref{sec:test_in_case_of_copula} we
have seen that one may choose $M(n) = \set{1,\dots,n}$ if $F$ is a copula
itself.

Since $F$ is continuous, transforming each $X_r^{(i)}$ by its df $F_r$ does
not alter the value of $n_{j,M(n)}(c)$ with probability one:
\begin{equation*}
  n_{j,M(n)}(c)
  = m_n - \sum_{i\in M(n)}  1_{\left[\bfzero, \left(1-\frac{c}{j}\right)\bfone\right]}\bigl(\bfU^{(i)}\bigr)
\end{equation*}
where $\bfU^{(i)} = \bigl(U_1^{(i)},\dots,U_d^{(i)}\bigr)^\T$,
$U_r^{(i)}:=F_r\bigl(X_r^{(i)}\bigr)$, and $\bfU^{(1)},\dots,\bfU^{(n)}$ are
iid with df $C_F$. As the margins of $F$ are typically unknown in
applications, we now replace $F_1,\dots,F_d$ with their empirical
counterparts $\hat F_{n,r}(x):=n^{-1}\sum_{i=1}^n 1_{(-\infty,x]}(X_r)$,
$x\in\R$, $1\le r\le d$, and obtain analogously
\begin{align*}
  \hat n_{j,M(n)}(c)
  &:= \sum_{i\in M(n)} 1_{(0,\infty)}\left(\sum_{r=1}^d 1_{\left(\hat F_{n,r}^{-1}\left(1-\frac{c}{j}\right),\infty\right)}\bigl(X_r^{(i)}\bigr)\right)  \\
  &\hphantom{:}= m_n - \sum_{i\in M(n)} 1_{\times_{r=1}^d\left[\vphantom{\frac{c}j}\smash{0,U_{\langle n(1-\frac{c}{j})\rangle:n,r}}\right]} \bigl(\bfU^{(i)}\bigr)
\end{align*}
with probability one. Note that
\begin{equation*} 
\hat F_{n,r}^{-1}\left(1-\frac{c}{j}\right) = X_{\neins:n,r},
\end{equation*}
where $\langle x\rangle:=\min\set{k\in\N:\, k\ge x}$ and $X_{1:n,r}\le
X_{2:n,r}\le\dots\le X_{n:n,r}$ denote the ordered values of
$X_r^{(1)},\dots, X_r^{(n)}$ for each $r=1,\dots,d$. Thus $U_{i:n,r} =
F_r\bigl(X_{i:n,r}\bigr)$ and $X_{i:n,r} = F_r^{-1}\bigl(U_{i:n,r}\bigr)$,
$1\le i\le n$, almost surely.

Since $\bfU^{(1)},\dots,\bfU^{(n)}$ are iid with df $C_F$, the distribution
of $\left(\hat n_{1,M(n)}(c),\dots,\hat n_{k,M(n)}(c)\right)^\T$ does not
depend on the marginal df $F_r$ but only on the copula $C_F$ of the
continuous df $F$. The following auxiliary result assures that we may
actually consider $\hat n_{j,M(n)}(c)$ instead of $n_{j,M(n)}(c)$.

\begin{lemma}\label{lem:crucial_approximation_of_n_j}
Suppose that $m_n\to\infty$, $m_n \log(m_n)/n\to 0$ as $n\to\infty$. Let
$c_n>0$ satisfy $c_n\to 0$, $m_nc_n\to\infty$ as $n\to \infty$. Then we
obtain for $j=1,\dots, k$
\begin{equation*}
  (m_nc_n)^{-1/2} \left(n_{j,M(n)}(c_n)-\hat n_{j,M(n)}(c_n)\right)\to_{n\to\infty} 0\quad\text{ in probability}.
\end{equation*}
\end{lemma}

\begin{proof}
We have almost surely
\begin{align*}
  &n_{j,M(n)}(c_n)-\hat n_{j,M(n)}(c_n)\\
  &= \sum_{i\in M(n)}
     \left(1_{\times_{r=1}^d\left[\vphantom{\frac{c_n}{j}}\smash{0,U_{\langle n(1-\frac{c_n}{j})\rangle:n,r}}\right]} \bigl(\bfU^{(i)}\bigr)
           - 1_{\left[\bfzero, \left(1-\frac{c_n}{j}\right)\bfone\right]}\bigl(\bfU^{(i)}\bigr) \right) \\
  &= \sum_{i\in M(n)} 1_{\times_{r=1}^d\left[\vphantom{\frac{c_n}{j}}\smash{0,U_{\langle n(1-\frac{c_n}{j})\rangle:n,r}}\right]} \bigl(\bfU^{(i)}\bigr) \left(1 - 1_{\left[\bfzero, \left(1-\frac{c_n}{j}\right)\bfone\right]}\bigl(\bfU^{(i)}\bigr) \right) \\
  &\mathrel{\hphantom{=}}{} - \sum_{i\in M(n)} 1_{\left[\bfzero, \left(1-\frac{c_n}{j}\right)\bfone\right]}\bigl(\bfU^{(i)}\bigr) \left(1 - 1_{\times_{r=1}^d[0,U_{\langle n(1-\frac{c_n}{j})\rangle:n,r}]} \bigl(\bfU^{(i)}\bigr)\right) \\
 &=: R_n-T_n.
\end{align*}
In what follows we show
\begin{equation*}
\frac 1{(m_nc_n)^{1/2}} E\left(  \sum_{i\in M(n)} 1_{\left(\vphantom{\frac{c_n}{j}}\smash{1-\frac{c_n}{j},U_{\langle n(1-\frac{c_n}{j})\rangle:n,r}}\right]}\bigl(U_r^{(i)}\bigr)\right)
= o(1)
\end{equation*}
and thus $(m_nc_n)^{-1/2}R_n=o_P(1)$; note that
\begin{align*}
  R_n\le \sum_{i\in M(n)}\sum_{r=1}^d 1_{\left(\vphantom{\frac{c_n}{j}}\smash{1-\frac{c_n}{j},U_{\langle n(1-\frac{c_n}{j})\rangle:n,r}}\right]}\bigl(U_r^{(i)}\bigr).
\end{align*}

Put $\eps_n:=\delta_n c_n^{1/2}/m_n^{1/2}$ with $\delta_n :=
3(m_n\log(m_n)/n)^{1/2}$. Then we have with $\mu_n:=
\left.E\left(U_{\neinsn:n,r}\right)=\neinsn\right/(n+1)$
\begin{align*}
  &\frac{m_n^{1/2}}{c_n^{1/2}} P\left( 1-\frac{c_n}{j}< U_r^{(1)} \le  U_{\neinsn:n,r}\right)\\
  &\le \frac{m_n^{1/2}}{c_n^{1/2}} P\left(1-\frac{c_n}{j} < U_r^{(1)}\le \mu_n+\eps_n\right)+ \frac{m_n^{1/2}}{c_n^{1/2}} P\left(U_{\neinsn:n,r}- \mu_n\ge \eps_n\right),
\end{align*}
where the first term is of order $O\left((nc_n)^{-1/2}+\delta_n\right)=o(1)$;
recall that $U_r^{(1)}$ is uniformly distributed on $(0,1)$ and $\delta_n\to
0$ as $n\to\infty$. Furthermore we deduce from \citet[Lemma 3.1.1]{reiss89}
the exponential bound
\begin{equation*}
  P\left(U_{\neinsn:n,r}- \mu_n\ge \eps_n\right)\le \exp\left(- \frac{\frac n{\sigma_n^2} \eps_n^2}{3\left(1+\frac{\eps_n}{\sigma_n^2}\right)}\right),
\end{equation*}
where $\sigma_n^2:=\mu_n(1-\mu_n)$ and
\begin{align*}
\frac{\eps_n}{\sigma_n^2} &= \delta_n \frac{c_n^{1/2}}{m_n^{1/2}} \frac{n+1}{\neinsn} \frac{n+1}{n+1-\neinsn}\\
&\le \delta_n \frac{c_n^{1/2}}{m_n^{1/2}} \frac{n+1}{\neinsn} \frac{j(n+1)}{nc_n}
=O\left(\frac{\delta_n}{(m_nc_n)^{1/2}}\right)
= o(1)
\end{align*}
as well as
\begin{align*}
\frac{n}{\sigma_n^2}\eps_n^2 &= 9c_n\log(m_n) \frac{n+1}{\neinsn} \frac{n+1}{n+1-\neinsn}\\
&\ge 9\log(m_n) \frac{n+1}{\neinsn} \frac {nc_n}{1+\frac{nc_n}{j}}.
\end{align*}
as $n\to\infty$. This implies
\begin{align*}
\frac{m_n^{1/2}}{c_n^{1/2}} P\left(U_{\neinsn:n,r}- \mu_n\ge \eps_n\right)
&\le 
\frac1{(m_n c_n)^{1/2}} \exp\left(- \frac18 \log(m_n)\right)
=o(1).
\end{align*}

Repeating the above arguments shows that $(m_nc_n)^{-1/2}T_n=o_P(1)$ as well,
which completes the proof of Lemma \ref{lem:crucial_approximation_of_n_j}.
\end{proof}

The previous result suggests a modification of our test statistic in
\eqref{eqn:definition_of_test_statistic_T_n}
\begin{equation*}
  \hat T_n(c):=\frac{\sum_{j=1}^k\left(j\,\hat n_{j,M(n)}(c) - \frac 1k\sum_{\ell=1}^k \ell\,\hat n_{\ell,M(n)}(c)\right)^2}{\frac 1k\sum_{\ell=1}^k \ell\,\hat n_{\ell,M(n)}(c)}
\end{equation*}
which does not depend on the margins but only on the copula of the underlying
df $F$. The following result is a consequence of Theorem
\ref{thm:limit_distribution_of_T_n} and Lemma
\ref{lem:crucial_approximation_of_n_j}.

\begin{theorem}\label{thm:limit_distribution_of_test_statistic_general}
Suppose that the df $F$ is continuous and that its copula $C_F$ satisfies
expansion \eqref{eqn:def_delta-neighborhood} for some $\delta>0$. Let
$m_n=\abs{M(n)}$ satisfy $m_n\to\infty$, $m_n\log(m_n)/n\to 0$ as
$n\to\infty$, and let $c=c_n$ satisfy $c_n\to 0$, $m_nc_n\to\infty$,
$m_nc_n^{1+2\delta}\to 0$ as $n\to\infty$. Then we obtain
\[
\hat T_n(c_n)\to_D \sum_{i=1}^{k-1} \lambda_i\xi_i^2
\]
with $\xi_i$ and $\lambda_i$ as in Theorem
\ref{thm:limit_distribution_of_T_n}.
\end{theorem}

The condition $m_nc_n^{1+2\delta}\to 0$ can again be dropped if the copula $C_F$ is a GPC.

\section{Testing for $\delta$-neighborhoods of a GPCP}\label{sec:test_GPCP}

In this section we carry the results of Section \ref{sec:test_multivariate}
over to function space, namely the space $C[0,1]$ of continuous functions on
$[0,1]$. A stochastic process $\bfeta=(\eta_t)_{t\in[0,1]}$ with sample paths
in $C[0,1]$ is called a \emph{standard max-stable process} (SMSP), if
$P(\eta_t\le x)=\exp(x)$, $x\le 0$, for each $t\in[0,1]$, and if the
distribution of $n\max_{1\le i\le n}\bfeta^{(i)}$ equals that of $\bfeta$ for
each $n\in\N$, where $\bfeta^{(1)},\bfeta^{(2)},\dots$ are independent copies
of $\bfeta$. All operations on functions such as $\max$, $+$, $/$ etc. are
meant pointwise. To improve the readability we set stochastic processes such
as $\bfeta$ or $\bfZ$ in bold font and deterministic functions like $f$ in
default font.

From \citet{ginhv90} --- see also \citet{aulfaho11} as well as \citet{hofm12}
--- we know that a stochastic process $\bfeta\in C[0,1]$ is an SMSP iff there exists a \emph{generator process} $\bfZ=(Z_t)_{t\in[0,1]}\in
C[0,1]$ with $0\le Z_t\le q$, $t\in[0,1]$, for some $q\in\R$, and $E(Z_t)=1$,
such that
\begin{equation*}
  P(\bfeta\le f)
  = \exp\left(-E\left(\sup_{t\in[0,1]}\left(\abs{f(t)}Z_t\right)\right)\right),\qquad f\in E^-[0,1].
\end{equation*}
By $E[0,1]$ we denote the set of those functions $f:[0,1]\to\R$, which are
bounded and have only a finite number of discontinuities; $E^-[0,1]$ is the
subset of those functions in $E[0,1]$ that attain only non positive values.
Note that
\[
\norm f_D:= E\left(\sup_{t\in[0,1]}\left(\abs{f(t)}Z_t\right)\right),\qquad f\in E[0,1],
\]
defines a norm on $E[0,1]$.

Let $\bfU=(U_t)_{t\in[0,1]}$ be a \emph{copula process}, i.e., each component
$U_t$ is uniformly distributed on $(0,1)$. A copula process $\bfU\in C[0,1]$
is said to be in the \emph{functional max-domain attraction} of an SMSP
$\bfeta$, denoted by $\bfU\in\mathcal D(\bfeta)$, if
\begin{equation}\label{def:functional_max-domain_of_attraction}
  P\left(n(\bfU-1_{[0,1]})\le f\right)^n\to_{n\to\infty} P(\bfeta\le f)=\exp\left(-\norm f_D\right),\qquad f\in E[0,1],
\end{equation}
where $1_{[0,1]}$ denotes the indicator function of the interval $[0,1]$.
This is the functional version of \eqref{def:domain_of_attraction}. We refer
to \citet{aulfaho11} for details. A more restrictive definition of functional
max-domain of attraction of stochastic processes in terms of usual weak
convergence was introduced by \citet{dehal01}.

From \citet[Proposition 8]{aulfaho11} we know that condition
\eqref{def:functional_max-domain_of_attraction} is equivalent with the
expansion
\begin{equation*} 
  P\left(\bfU\le 1_{[0,1]}+cf\right)=1- c\norm f_D+o(c), \qquad f\in E^-[0,1],
\end{equation*}
as $c\downarrow 0$. In particular we obtain in this case
\begin{equation}\label{eqn:expansion_of_df_of_copula_process}
  P\left(\bfU\le (1-c)1_{[0,1]}\right)=1-c\left(m_D+r(-c)\right),
\end{equation}
where $m_D:=E\left(\sup_{t\in[0,1]}Z_t\right) = \norm{1_{[0,1]}}_D$ is the
uniquely determined \emph{generator constant} pertaining to $\bfeta$, and the
remainder term satisfies $r(-c)\to 0$ as $c\downarrow 0$.

A copula process $\bfV\in C[0,1]$ is a \emph{generalized Pareto copula
process} (GPCP), if there exists $\eps_0>0$ such that
\begin{equation*} 
  P\left(\bfV\le 1_{[0,1]}+f\right)=1- \norm f_D, \qquad f\in E^-[0,1],\,\norm f_\infty\le \eps_0.
\end{equation*}
A GPCP with a prescribed $D$-norm $\norm\cdot_D$ can easily be constructed;
cf. \citet[Example 5]{aulfaho11}. Its characteristic property is its
excursion stability; cf. \citet{dehaf06}. A copula process $\bfU\in C[0,1]$,
consequently, satisfies $\bfU\in\mathcal D(\bfeta)$ iff there exists a GPCP
$\bfV$ such that
\begin{equation}\label{eqn:equivalent_condition_functional_max-domain_attraction_via_GPCP}
P\left(\bfU\le 1_{[0,1]}+cf\right)=P\left(\bfV\le 1_{[0,1]}+cf\right)+o(c), \qquad f\in E^-[0,1],
\end{equation}
as $c\downarrow 0$. If the remainder term $o(c)$ in expansion
\eqref{eqn:equivalent_condition_functional_max-domain_attraction_via_GPCP} is
in fact of order $O(c^{1+\delta})$ for some $\delta>0$, then the copula
process $\bfU\in C[0,1]$ is said to be in the \emph{$\delta$-neighborhood} of
a GPCP; cf. \eqref{eqn:def_delta-neighborhood}.

\subsection{Observing copula processes}

The test statistic $T_n(c_n)$ investigated in Section
\ref{sec:test_in_case_of_copula} carries over to function space $C[0,1]$,
which enables us to check whether a given copula process
$\bfU=(U_t)_{t\in[0,1]}$ is in a $\delta$-neighborhood of an SMSP $\bfV$. Put
for $s<0$
\begin{equation*}
S_{\bfU}(s):=\int_0^1 1_{(s,\infty)}\left(U_t-1\right)\,dt\in[0,1],
\end{equation*}
which is the sojourn time that the process $\bfU$ spends above the threshold $1+s$. If $\bfU\in\mathcal D(\bfeta)$, then we obtain from equation \eqref{eqn:expansion_of_df_of_copula_process}
\begin{align*}\label{eqn:expansion_of_binomial_probability_functional_case}
P\left(S_{\bfU}(s)>0\right)&=1- P\left(S_{\bfU}(s)=0\right)\nonumber\\
&=1-P\left(\bfU\le (1+s)1_{[0,1]}\right)\nonumber\\
&=\abs s\left(m_D+ r(s)\right).
\end{align*}

Choose again $k\in\N$, $k\ge2$, and put for $j=1,\dots,k$ and $c>0$
\begin{equation*}
  n_j(c):=\sum_{i=1}^n 1_{(0,1]}\left(S_{\bfU^{(i)}}\left(-\frac{c}j\right)\right)
\end{equation*}
where $\bfU^{(1)},\dots,\bfU^{(n)}$ are independent copies of $\bfU$. Then
$n_j(c)$ is the number of those processes among
$\bfU^{(1)},\dots,\bfU^{(n)}$, which exceed the threshold $1-\frac{c}j$ in at
least one point.

If $\bfU\in\mathcal D(\bfeta)$, then each $n_j(c)$ is binomial $B(n,p_j(c))$ distributed with
\begin{equation*}
p_j(c)=P\left(S_{\bfU}\left(-\frac{c}j\right)>0\right) = \frac{c}j\left(m_D + r\left(-\frac{c}j\right)\right).
\end{equation*}
Put again
\begin{equation} \label{eqn:test_statistic_copula_process}
  T_n(c):= \frac{\sum_{j=1}^k\left(j\,n_j(c) - \frac 1k\sum_{\ell=1}^k \ell\,n_\ell(c)\right)^2}{\frac 1k\sum_{\ell=1}^k \ell\,n_\ell(c)}.
\end{equation}
Repeating the arguments in the proof of Theorem \ref{thm:limit_distribution_of_T_n}, one shows that its assertion carries over to the functional space as well.

\begin{theorem}\label{thm:limit_distribution_of_T_n_for_functional_data}
Suppose that the copula process $\bfU\in C[0,1]$ is in the $\delta$-neighborhood of
a GPCP for some $\delta>0$. In this case the remainder term $r(s)$  in expansion
\eqref{eqn:expansion_of_df_of_copula_process} is
of order $O\bigl(\abs s^\delta\bigr)$ as $s\to 0$. Let
$c=c_n$ satisfy $c_n\to 0$, $nc_n\to\infty$ and $nc_n^{1+2\delta}\to 0$ as
$n\to\infty$. Then we obtain
\[
T_n(c_n)\to_D  \sum_{i=1}^{k-1} \lambda_i \xi_i^2,
\]
with $\xi_i$ and $\lambda_i$ as in Theorem \ref{thm:limit_distribution_of_T_n}.
\end{theorem}

\subsection{The case of more general processes}\label{sec:general_processes}

In what follows we will extend Theorem
\ref{thm:limit_distribution_of_T_n_for_functional_data} to the case when
observing the underlying copula process is subject to a certain kind of
nuisance. Let $\bfX=(X_t)_{t\in[0,1]}\in C[0,1]$ be a stochastic process with
identical continuous univariate marginal df, i.e., $F(x):=P(X_0\le
x)=P(X_t\le x)$, $t\in[0,1]$, is a continuous function in $x\in\R$. $\bfX$ is
said to be in the functional max-domain of attraction of a max-stable process
$\bfxi=(\xi_t)_{t\in[0,1]}$, if the copula process
$\bfU=(F(X_t))_{t\in[0,1]}$ satisfies $\bfU\in\mathcal D(\bfeta)$, where
$\bfeta$ is a SMSP, and the df $F$ satisfies the univariate extreme value
condition; for the univariate case we refer to \citet[Section 2.1]{fahure10},
among others.

Let $\bfX^{(1)},\dots,\bfX^{(n)}$ be independent copies of the process $\bfX$
and denote the sample df pertaining to the univariate iid observations
$X_0^{(1)},\dots,X_0^{(n)}$ by $\hat F_n(x):=n^{-1}\sum_{i=1}^n
1_{(-\infty,x]}\bigl(X_0^{(i)}\bigr)$, $x\in\R$. As in Section
\ref{sec:arbitrary_random_vector}, fix $k\in\set{2,3,\dots}$, choose an
arbitrary subset $M(n)$ of $\set{1,\dots,n}$ of size $m(n)=\abs{M(n)}$, and
put for $c>0$
\begin{align*}
  n_{j,M(n)}(c)
  &:=\sum_{i\in M(n)}1_{(0,1]}\left(\int_0^1 1_{(\gamma(c),\infty)}\bigl(X_t^{(i)}\bigr)\,dt \right) \\
  &\hphantom{:}= \sum_{i\in M(n)} 1_{[0,1)}\left(\int_0^1 1_{(-\infty,\gamma(c)]}\bigl(X_t^{(i)}\bigr)\,dt \right) \\
  &\hphantom{:}= \sum_{i\in M(n)} 1_{[0,1)}\left(\int_0^1 1_{\left[0,1-\frac{c}{j}\right]}\bigl(U_t^{(i)}\bigr)\,dt \right)\\
  &= \sum_{i\in M(n)} 1_{\set{\bfU^{(i)}\nleq (1-\frac{c}{j})1_{[0,1]}}}
\end{align*}
where $\gamma(c) := F^{-1}\bigl(1-\frac{c}{j}\bigr)$ and the next to last
equation holds almost surely. Again, we replace the marginal df $F$ with its
empirical counterpart and obtain analogously with $\hat\gamma_n(c) := \hat
F_n^{-1}\bigl(1-\frac{c}{j}\bigr)$
\begin{equation*}
  \hat n_{j,M(n)}(c)
  := \sum_{i\in M(n)} 1_{(0,1]}\left(\int_0^1 1_{(\hat\gamma_n(c),\infty)}\bigl(X_t^{(i)}\bigr)\,dt \right)
\end{equation*}
Thus the rv $\hat n_j(c)$ is the total number of processes
$\bfX^{(i)}=\bigl(X^{(i)}_t\bigr)_{t\in[0,1]}$ among
$\bfX^{(1)},\dots,\bfX^{(i)}$, which exceed the random threshold $\hat
F_n^{-1}\bigl(1-\frac{c}{j}\bigr)$ for some $t\in[0,1]$. Note that the
distribution of the rv $\left(\hat n_1(c),\dots,\hat n_k(c)\right)^{\T}$ does
not depend on $F$ but on the copula process $\bfU$ since
\begin{align*}
  \hat n_{j,M(n)}(c)
  &\hphantom{:}= \sum_{i\in M(n)} 1_{[0,1)}\left(\int_0^1 1_{(-\infty,\hat\gamma_n(c)]}\bigl(X_t^{(i)}\bigr)\,dt \right) \\
  &\hphantom{:}= \sum_{i\in M(n)} 1_{[0,1)}\left(\int_0^1 1_{[0,U_{\langle n(1-\frac{c}{j})\rangle:n}]}\bigl(U_t^{(i)}\bigr)\,dt \right) \\
  &\hphantom{:}= \sum_{i\in M(n)} 1_{\bigl\{\bfU^{(i)}\nleq U_{\langle n(1-\frac{c}{j})\rangle:n}1_{[0,1]}\bigr\}}
\end{align*}
with probability one, where $U_{1:n}\le\dots\le U_{n:n}$ denote the ordered
values of $U_0^{(1)},\dots,U_0^{(n)}$. The following auxiliary result is the
extension of Lemma \ref{lem:crucial_approximation_of_n_j} to function space.

\begin{lemma}\label{lem:crucial_approximation_function_space}
Suppose that the copula process $\bfU\in C[0,1]$ corresponding to $\bfX$ is
in the $\delta$-neighborhood of a GPCP for some $\delta>0$. In this case the
remainder term $r(s)$ in expansion
\eqref{eqn:expansion_of_df_of_copula_process} is of order $O(\abs{s}^\delta)$
as $s\uparrow 0$. Choose $M(n)\subset\set{1,\dots,n}$ and $c_n>0$ such that
$m_n\to\infty$, $m_n\log(m_n) /n\to 0$, $c_n\to 0$, $m_nc_n\to\infty$ and
$m_nc_n^{1+2\delta}\to 0$ as $n\to \infty$.  Then we obtain for $j=1,\dots,
k$
\begin{equation*}
(m_nc_n)^{-1/2}\left(n_{j,M(n)}(c_n)-\hat n_{j,M(n)}(c_n)\right)\to_{n\to\infty}0\quad\text{in probability}.
\end{equation*}
\end{lemma}

\begin{proof}
We have with probability one
\begin{align*}
  &n_{j,M(n)}(c_n)-\hat n_{j,M(n)}(c_n) \\
  &= \sum_{i\in M(n)}\left(
        1_{\set{\bfU^{(i)}\nleq (1-\frac{c_n}{j})1_{[0,1]}}}
        - 1_{\bigl\{\bfU^{(i)}\nleq U_{\langle n(1-\frac{c_n}{j})\rangle:n}1_{[0,1]}\bigr\}}
     \right) \\
  &= \sum_{i\in M(n)} 1_{\bigl\{\bfU^{(i)}\nleq (1-\frac{c_n}{j})1_{[0,1]},\ \bfU^{(i)}\le U_{\langle n(1-\frac{c_n}{j})\rangle:n}1_{[0,1]}\bigr\}} \\
  &\mathrel{\hphantom{=}}{} - \sum_{i\in M(n)} 1_{\bigl\{\bfU^{(i)}\nleq U_{\langle n(1-\frac{c_n}{j})\rangle:n}1_{[0,1]},\ \bfU^{(i)}\le (1-\frac{c_n}{j})1_{[0,1]}\bigr\}} \\
  &=: R_n-T_n.
\end{align*}
We show in what follows that
\begin{equation} \label{eqn:expectation_of_R_n}
  \frac 1{(m_nc_n)^{1/2}} E(R_n)
  \to_{n\to\infty} 0,
\end{equation}
proceeding as in the proof of Lemma \ref{lem:crucial_approximation_of_n_j}:
Put $\eps_n:=\delta_n c_n^{1/2}/m_n^{1/2}$ with $\delta_n :=
3(m_n\log(m_n)/n)^{1/2}$. Note that $\mu_n:=
E\left(U_{\neinsn:n,r}\right)=\left.\neinsn\right/(n+1)$ satisfies
\begin{equation*}
  1-\mu_n-\eps_n
  \ge \frac{n}{n+1} \frac{c_n}{j} - \eps_n
  = \frac{c_n^{1/2}}{m_n^{1/2}} \left(\frac{n}{n+1}\frac{(m_nc_n)^{1/2}}{j} - \delta_n\right) > 0
\end{equation*}
for large values of $n$ as well as
\begin{equation*}
  1-\mu_n+\eps_n
  \le \frac1{n+1} + \frac{n}{n+1}\frac{c_n}{j} + \eps_n
  = \frac{c_n}{j} + \eps_n + \frac1{n+1}\left(1-\frac{c_n}{j}\right)
  = O(c_n)
\end{equation*}
Now we obtain by expansion \eqref{eqn:expansion_of_df_of_copula_process}, if
$n$ is sufficiently large,
\begin{align*}
  &P\left( \bfU^{(1)}\nleq \left(1-\frac{c_n}{j}\right)1_{[0,1]},\ \bfU^{(1)}\le U_{\langle n(1-\frac{c_n}{j})\rangle:n}1_{[0,1]} \right) \\
  &= P\left( \bfU^{(1)}\le U_{\langle n(1-\frac{c_n}{j})\rangle:n}1_{[0,1]} \right)
     - P\Bigl( \bfU^{(1)}\le \min\Bigl\{1-\frac{c_n}{j}, U_{\langle n(1-\frac{c_n}{j})\rangle:n}\Bigr\}1_{[0,1]} \Bigr) \\
  &\le P\left( \bfU^{(1)}\le (\mu_n + \eps_n)1_{[0,1]} \right)
     + P\left( U_{\langle n(1-\frac{c_n}{j})\rangle:n} \ge \mu_n+\eps_n\right) \\
  &\mathrel{\hphantom{\le}}{}
     - P\Bigl( \bfU^{(1)}\le \min\Bigl\{1-\frac{c_n}{j}, \mu_n-\eps_n\Bigr\}1_{[0,1]} \Bigr)
     + P\left(U_{\langle n(1-\frac{c_n}{j})\rangle:n} \le \mu_n-\eps_n\right) \\
  &= P\left(\abs{U_{\langle n(1-\frac{c_n}{j})\rangle:n} - \mu_n}\ge\eps_n\right)
     - (1-\mu_n-\eps_n)\left(m_D+r(\mu_n+\eps_n-1)\right) \\
  &\mathrel{\hphantom{=}}{}
     + \max\Bigl\{\frac{c_n}{j}, 1-\mu_n+\eps_n\Bigr\} \left(m_D + r\left(-\max\Bigl\{\frac{c_n}{j}, 1-\mu_n+\eps_n\Bigr\}\right)\right) \\
  &\le P\left(\abs{U_{\langle n(1-\frac{c_n}{j})\rangle:n} - \mu_n}\ge\eps_n\right)
     - \left(\frac{n}{n+1} \frac{c_n}{j} - \eps_n\right)\left(m_D + O\left(c_n^\delta\right)\right) \\
  &\mathrel{\hphantom{=}}{}
     + \left(\frac1{n+1} + \frac{n}{n+1}\frac{c_n}{j} + \eps_n\right) \left(m_D + O\left(c_n^\delta\right)\right) \\
  &= P\left(\abs{U_{\langle n(1-\frac{c_n}{j})\rangle:n} - \mu_n}\ge\eps_n\right)
     + O\left(c_n^{1+\delta}\right) + O\left(\frac1n + \eps_n\right).
\end{align*}
The arguments in the proof of Lemma \ref{lem:crucial_approximation_of_n_j}
show
\begin{equation*}
  \frac{m_n^{1/2}}{c_n^{1/2}} P\left(\abs{U_{\langle n(1-\frac{c_n}{j})\rangle:n} - \mu_n}\ge\eps_n\right) = o(1)
\end{equation*}
as $n\to\infty$ and, thus, \eqref{eqn:expectation_of_R_n}; recall
$m_nc_n^{1+2\delta} = o(1)$ and note that
\begin{equation*}
  \frac{m_n^{1/2}}{c_n^{1/2}} \left(\frac1n + \eps_n\right)
  = \frac{m_n^{1/2}}{n^{1/2}} \frac1{(nc_n)^{1/2}} + \delta_n
  = o(1).
\end{equation*}

Repeating the above arguments one shows that  $E(T_n)=o((m_nc_n)^{1/2})$ as
$n\to\infty$ as well, which completes the proof of Lemma
\ref{lem:crucial_approximation_function_space}.
\end{proof}

Analogously to Section \ref{sec:test_multivariate} we now choose
$M(n)\subset\set{1,\dots,n}$ and replace $n_j(c)$ in
\eqref{eqn:test_statistic_copula_process} with $\hat n_{j,M(n)}(c)$ and
obtain
\begin{equation*}
  \hat T_n(c)
  := \frac{\sum_{j=1}^k \left(j\,\hat n_{j,M(n)}(c)-\frac 1k \sum_{\ell=1}^k \ell\,\hat n_{\ell,M(n)}(c)\right)^2} {\frac1k \sum_{\ell=1}^k \ell\,\hat n_{\ell,M(n)}(c)}.
\end{equation*}
By this statistic we can in particular check, whether the copula process
$\bfU=(F(X_t))_{t\in[0,1]}$ pertaining to $\bfX$ is in a $\delta$-neigborhood
of some GPCP $\bfV$. Its distribution does not depend on the marginal df $F$
of $\bfX$ but on the copula process $\bfU$. The next result follows from
Lemma \ref{lem:crucial_approximation_function_space} and the arguments in the
proof of Theorem \ref{thm:limit_distribution_of_T_n}.

\begin{theorem}
We have under the conditions of Lemma \ref{lem:crucial_approximation_function_space}
\begin{equation*}
\hat T_n(c_n)\to_D  \sum_{i=1}^{k-1} \lambda_i \xi_i^2,
\end{equation*}
with $\xi_i$ and $\lambda_i$ as in Theorem \ref{thm:limit_distribution_of_T_n}.
\end{theorem}

\begin{exam}\upshape
Let $\eta_1,\eta_2$ be two independent and standard negative exponential
distributed rv. Put for $t\in[0,1]$
\begin{equation*}
  X_t
  :=\max\left(\frac {-V}{2\exp\bigl(\frac{\eta_1}{1-t}\bigr)}, \frac {-V}{2\exp\left(\frac{\eta_2}t\right)}\right)
  =\frac{-V}{2\exp\left(\max\left(\frac{\eta_1}{1-t},\frac{\eta_2}t\right)\right)},
\end{equation*}
where the rv $V$ is independent of $\eta_1,\eta_2$ and follows the df
$H_\lambda$ defined in Lemma \ref{lem:copula_not_in_domain_of_attraction}
with $\lambda\in\bigl[-\frac{\sqrt2}2,\frac{\sqrt2}2\bigr]$. Note that
\[
\max\left(\frac{\eta_1}{1-t},\frac{\eta_2}t\right)=_D\eta_1=_D\eta_2
\]
and, thus, the process $\bfX=(X_t)_{t\in[0,1]}$ has identical continuous
marginal df.

For $\lambda=0$, the process $\bfX$ is a \emph{generalized Pareto process},
whose pertaining copula process is in the max-domain of attraction of a SMSP,
see \citet{aulfaho11}. For $\lambda\not=0$ this is not true: Just consider
the bivariate rv $(X_0,X_1)=_D-\frac{V}{2}(1/U_1,1/U_2)$, where
$U_1=\exp(\eta_1)$, $U_2=\exp(\eta_2)$ and repeat the arguments in Lemma
\ref{lem:copula_not_in_domain_of_attraction}.
\end{exam}

\section{Testing via a grid of points} \label{sec:test_processes_grid}
Observing a complete process on $[0,1]$ as in the preceding section might be
a too restrictive assumption. Instead we will require in what follows that we
observe stochastic processes with sample paths in $C[0,1]$ only through an
increasing grid of points in $[0,1]$.

Let $\bfV\in C[0,1]$ be a GPCP with pertaining $D$-norm
\begin{equation*}
  \norm f_D =E\left(\sup_{t\in[0,1]}(\abs{f(t)}Z_t)\right),\qquad f\in E[0,1].
\end{equation*}
Choose a grid of points $0=t_1^{(d)}<t_2^{(d)}<\dots<t_d^{(d)}=1$. Then the rv
\begin{equation*}
  \bfV_d := \left(V_{t_1}^{(d)},\dots,V_{t_d}^{(d)}\right)^\intercal
\end{equation*}
follows a GPC, whose corresponding $D$-norm is given by
\begin{equation*}
  \norm{\bfx}_{D,d}:=E\left(\max_{1\le i\le d}\left(\abs{x_i}Z_{t_i}\right)\right),\qquad \bfx\in\R^d.
\end{equation*}

Let now $d=d_n$ depend on $n$. If we require that
\begin{equation*}
  \max_{1\le i\le d_n-1}\abs{t_{i+1}^{(d_n)} -  t_i^{(d_n)}}\to_{n\to\infty} 0,
\end{equation*}
then, by the continuity of $\bfZ=(Z_t)_{t\in[0,1]}$,
\begin{equation*}
  \max_{1\le i\le d_n-1}\abs{Z_{t_{i+1}^{(d_n)}} - Z_{t_i^{(d_n)}}}\to_{n\to\infty} 0,\quad
  \max_{1\le i\le d_n} Z_{t_i^{(d_n)}} \to_{n\to\infty} \sup_{t\in[0,1]} Z_t \qquad \mathrm{a.s.},
\end{equation*}
and, thus, the sequence of generator constants converges:
\begin{equation*}
  m_{D,d_n}
  := E\left(\max_{1\le i\le d_n} Z_{t_i^{(d_n)}}\right)
  \to_{n\to\infty} E\left(\sup_{t\in[0,1]}Z_t\right)
  = m_D.
\end{equation*}

\subsection{Observing copula data} \label{sec:test_GPCP_grid}

Suppose we are given $n$ independent copies of a copula process $\bfU$. The
projection of each process onto the grid
$0=t_1^{(d_n)}<t_2^{(d_n)}<\dots<t_{d_n}^{(d_n)}=1$ yields $n$ iid rv in
$\R^{d_n}$, which follow a copula. Let $n_j(c)$ as defined in
\eqref{eqn:definition_of_n_j(c)} be based on these rv. Note that $n_j(c)$
depends on $d_n$ as well. But in order not to overload our notation we
suppress the dependence on the dimension.

Moreover we require that
\begin{equation} \label{eqn:uniform_delta-neighborhood}
  P\left(\bfU\le 1_{[0,1]}+cf\right)
  = 1- c\norm f_D + O(c^{1+\delta})
\end{equation}
holds uniformly for all $f\in E^-[0,1]$ satisfying $\norm{f}_\infty\le1$.
Again a suitable version of the central limit theorem implies
\begin{equation*}
  \frac 1{(nc_n)^{1/2}} \bigl(j\,n_j(c_n) - n c_n m_{D,d_n}\bigr)
  \to_D N\left(0,jm_D\right)
\end{equation*}
and thus
\begin{equation*}
  \frac{j\,n_j(c_n)}{nc_n}
  \to_{n\to\infty}m_D\quad\text{in probability},\qquad 1\le j\le
  k,
\end{equation*}
yielding
\begin{equation*}
  \frac 1{nc_nk} \sum_{j=1}^k j\,n_j(c_n)
  \to_{n\to\infty} m_D \quad\text{in probability}.
\end{equation*}
Theorem \ref{thm:limit_distribution_of_T_n} now carries over:

\begin{theorem}\label{thm:asymptotic_distribution_of_T_n_for_gpcp}
Let $\bfU$ be a copula process satisfying
\eqref{eqn:uniform_delta-neighborhood}. Choose a grid of points
$0=t_1^{(d)}<t_2^{(d)}<\dots<t_d^{(d)}=1$ with $d=d_n\to\infty$ and
$\max_{1\le i\le d_n-1}\abs{t_{i+1}^{(d_n)} - t_i^{(d_n)}}\to 0$ as
$n\to\infty$. Let $T_n$ as defined in
\eqref{eqn:definition_of_test_statistic_T_n} be based on the projections of
$n$ independent copies of $\bfU$ onto this increasing grid of points. Let
$c=c_n$ satisfy $c_n\to 0$, $nc_n\to\infty$ and $nc_n^{1+2\delta}\to0$ as
$n\to\infty$. Then we obtain
\begin{equation*}
T_n(c_n)\to_D  \sum_{i=1}^{k-1} \lambda_i \xi_i^2,
\end{equation*}
with $\xi_i$ and $\lambda_i$ as in Theorem
\ref{thm:limit_distribution_of_T_n}.
\end{theorem}

\subsection{The case of an arbitrary process} \label{sec:test_an_arbitrary_process} Now we will extend Theorem
\ref{thm:asymptotic_distribution_of_T_n_for_gpcp} to a general process
$\bfX=(X_t)_{t\in[0,1]}\in C[0,1]$ with continuous marginal df $F_t$,
$t\in[0,1]$. We want to test whether the copula process
$\bfU:=\left(F_t(X_t)\right)_{t\in[0,1]}\in C[0,1]$ corresponding to $\bfX$
satisfies \eqref{eqn:uniform_delta-neighborhood}. As before this will be done
by projecting the process $\bfX$ onto a grid of points
$0=t_1^{(d)}<\dots<t_d^{(d)}=1$ with $d=d_n\to_{n\to\infty}\infty$ and and
$\max_{1\le i\le d_n-1}\abs{t_{i+1}^{(d_n)} - t_i^{(d_n)}}\to_{n\to\infty}
0$.

Let $\bfX^{(1)},\dots,\bfX^{(n)}$ be independent copies of $\bfX$ and
consider the $n$ iid rv of projections $\bfX_{d_n}^{(i)} :=
\left(X_{t_1^{(d_n)}}^{(i)},\dots,X_{t_{d_n}^{(d_n)}}^{(i)}\right)^\T$,
$i=1,\dots,n$. Fix $k\in\set{2,3,\dots}$, choose an arbitrary subset $M(n)$
of $\set{1,\dots,n}$ of size $\abs{M(n)}=m_n$ and put for $0<c<1$
\begin{align*}
  n_{j,M(n)}(c)
  &:= \sum_{i\in M(n)} 1_{(0,\infty)}\left(\sum_{r=1}^{d_n} 1_{(\gamma_{j,r}(c),\infty)}\left(X_{t_r^{(d_n)}}^{(i)}\right)\right) \\
  &\hphantom{:}= m_n - \sum_{i\in M(n)} 1_{(-\bfinfty,\bfgamma_j(c)]}\left(\bfX_{d_n}^{(i)}\right)
\end{align*}
which is the number of all rv $\bigl(\bfX_{d_n}^{(i)}\bigr)_{i\in M(n)}$
exceeding the vector
\begin{equation*}
  \bfgamma_j(c)
  := \left(\gamma_{j,1}(c), \dots, \gamma_{j,d_n}(c)\right)^\T
  := \left(F_{t_1^{(d_n)}}^{-1}\left(1-\frac{c}{j}\right),\dots,F_{t_{d_n}^{(d_n)}}^{-1}\left(1-\frac{c}{j}\right)\right)^\T
\end{equation*}
in at least one component. Clearly we have
\begin{equation*}
  n_{j,M(n)}(c)
  = m_n - \sum_{i\in M(n)} 1_{\left[\bfzero, \left(1-\frac{c}{j}\right)\bfone\right]}\left(\bfU_{d_n}^{(i)}\right)
\end{equation*}
almost surely where $\bfU_{d_n}^{(i)} :=
\left(U_{t_1^{(d_n)}}^{(i)},\dots,U_{t_{d_n}^{(d_n)}}^{(i)}\right)^\T$ and
$\bfU_{d_n}^{(i)}$ is the copula process of $\bfX_{d_n}^{(i)}$.

Again we replace $\bfgamma_j(c)$ with
\begin{equation*}
  \hat\bfgamma_j(c)
  := \left(\hat\gamma_{j,1}(c), \dots, \hat\gamma_{j,d_n}(c)\right)^\T
\end{equation*}
where $\hat\gamma_{j,r}(c) := \hat
F_{t_r^{(d_n)}}^{-1}\left(1-\frac{c}{j}\right)$ and $\hat F_{t_r^{(d_n)}}(x)
:= n^{-1}\sum_{i=1}^n 1_{(-\infty,x]}\left(X_{t_r^{(d_n)}}^{(i)}\right)$,
$1\le r\le d_n$, yielding an estimator of $n_j(c)$:
\begin{align*}
  \hat n_{j,M(n)}(c)
  &= \sum_{i\in M(n)} 1_{(0,\infty)}\left(\sum_{r=1}^d 1_{\left(\hat\gamma_{j,r}(c),\infty\right)}\left(X_{t_r^{(d_n)}}^{(i)}\right)\right) \\
  &= m_n - \sum_{i\in M(n)} 1_{(-\bfinfty,\hat\bfgamma_j(c)]}\left(\bfX_{d_n}^{(i)}\right).
\end{align*}
We have
\begin{equation*}
  \hat\gamma_{j,r}(c) = X_{\neins:n,r},
\end{equation*}
where $X_{1:n,r}\le X_{2:n,r}\le\dots\le X_{n:n,r}$ denote the ordered values
of $X_{t_r^{(d_n)}}^{(1)},\dots, X_{t_r^{(d_n)}}^{(n)}$ for each
$r=1,\dots,d_n$ and $\langle x\rangle=\min\set{k\in\N:\, k\ge x}$ is again
the right integer neighbor of $x>0$. Since transforming each
$X_{t_r^{(d_n)}}^{(i)}$ by its df $F_{t_r^{(d_n)}}$ does not alter the value
of $\hat n_j(c)$ with probability one, we obtain
\begin{equation*}
  \hat n_{j,M(n)}(c)
  = m_n - \sum_{i\in M(n)} 1_{\times_{r=1}^{d_n}\left[\vphantom{\frac{c}j}\smash{0,U_{\langle n(1-\frac{c}{j})\rangle:n,r}}\right]} \bigl(\bfU_{d_n}^{(i)}\bigr)
\end{equation*}
almost surely where $U_{1:n,r}\le U_{2:n,r}\le\dots\le U_{n:n,r}$ are the
order statistics of $U_{t_r^{(d_n)}}^{(1)},\dots, U_{t_r^{(d_n)}}^{(n)}$.
Since $\bfU_{d_n}^{(1)},\dots,\bfU_{d_n}^{(n)}$ are independent copies of the
rv $\bfU_{d_n} :=
\left(F_{t_1^{(d_n)}}(X_{t_1^{(d_n)}}),\dots,F_{t_{d_n}^{(d_n)}}(X_{t_{d_n}^{(d_n)}})\right)^\T$,
the distribution of $\left(\hat n_{1,M(n)}(c),\dots,\hat
n_{k,M(n)}(c)\right)^\T$ does not depend on the marginal df $F_{t_r^{(d)}}$.
The following auxiliary result is crucial.

\begin{lemma}\label{newlem:crucial_approximation_of_n_j}
Suppose that $m_n\to_{n\to\infty}\infty$. Let $c_n>0$ satisfy $c_n\to 0$,
$m_nc_n\to\infty$ and $m_n^2\log(m_n)c_n/n\to 0$ as $n\to \infty$. If
$d_n\to\infty$ satisfies $d_n^2/(m_nc_n)\to 0$ as $n\to\infty$, then we
obtain for $j=1,\dots, k$
\begin{equation*}
  (m_nc_n)^{-1/2} \left(n_{j,M(n)}(c_n)-\hat n_j(c_n)\right)\to_{n\to\infty} 0\quad\mathrm{ in\  probability}.
\end{equation*}
\end{lemma}

\begin{proof}
We have almost surely
\begin{align*}
  &n_{j,M(n)}(c_n)-\hat n_{j,M(n)}(c_n)\\
  &= \sum_{i\in M(n)}
     \left(1_{\times_{r=1}^{d_n}\left[\vphantom{\frac{c_n}{j}}\smash{0,U_{\langle n(1-\frac{c_n}{j})\rangle:n,r}}\right]} \bigl(\bfU_{d_n}^{(i)}\bigr)
           - 1_{\left[\bfzero, \left(1-\frac{c_n}{j}\right)\bfone\right]}\bigl(\bfU_{d_n}^{(i)}\bigr) \right) \\
  &= \sum_{i\in M(n)} 1_{\times_{r=1}^{d_n}\left[\vphantom{\frac{c_n}{j}}\smash{0,U_{\langle n(1-\frac{c_n}{j})\rangle:n,r}}\right]} \bigl(\bfU_{d_n}^{(i)}\bigr) \left(1 - 1_{\left[\bfzero, \left(1-\frac{c_n}{j}\right)\bfone\right]}\bigl(\bfU_{d_n}^{(i)}\bigr) \right) \\
  &\mathrel{\hphantom{=}}{} - \sum_{i\in M(n)} 1_{\left[\bfzero, \left(1-\frac{c_n}{j}\right)\bfone\right]}\bigl(\bfU_{d_n}^{(i)}\bigr) \left(1 - 1_{\times_{r=1}^d[0,U_{\langle n(1-\frac{c_n}{j})\rangle:n,r}]} \bigl(\bfU_{d_n}^{(i)}\bigr)\right) \\
 &=: R_n-T_n.
\end{align*}
In what follows we show
\begin{equation*}
  \frac 1{(m_nc_n)^{1/2}} \sum_{r=1}^{d_n} E\left(  \sum_{i\in M(n)} 1_{\left(\vphantom{\frac{c_n}{j}}\smash{1-\frac{c_n}{j},U_{\langle n(1-\frac{c_n}{j})\rangle:n,r}}\right]}\left(U_{t_r^{(d_n)}}^{(i)}\right)\right)
  = o(1)
\end{equation*}
and thus $(m_nc_n)^{-1/2}R_n=o_P(1)$; note that
\begin{align*}
  R_n\le \sum_{i\in M(n)}\sum_{r=1}^{d_n} 1_{\left(\vphantom{\frac{c_n}{j}}\smash{1-\frac{c_n}{j},U_{\langle n(1-\frac{c_n}{j})\rangle:n,r}}\right]}\left(U_{t_r^{(d_n)}}^{(i)}\right).
\end{align*}

Put $\eps_n:=\delta_n c_n^{1/2}/m_n^{1/2}$ with
$\delta_n:=(m_nc_n)^{-1/2}\to_{n\to\infty}0$. We have with $\mu_n:=
\left.E\left(U_{\neinsn:n,r}^{(1)}\right)=\neinsn\right/(n+1)$
\begin{align*}
&\frac{m_n^{1/2}}{c_n^{1/2}} P\left( 1-\frac{c_n}{j}< U_{t_r^{(d_n)}}^{(1)} \le  U_{\neinsn:n,r}\right)\\
&\le \frac{m_n^{1/2}}{c_n^{1/2}} P\left(1-\frac{c_n}{j} < U_{t_r^{(d_n)}}^{(1)}\le \mu_n+\eps_n\right)+ \frac{m_n^{1/2}}{c_n^{1/2}} P\left(U_{\neinsn:n,r}- \mu_n\ge \eps_n\right)
\end{align*}
where the first term is of order
$O\left((nc_n)^{-1/2}+\delta_n\right)=O((m_nc_n)^{-1/2})$; recall that
$U_{t_r^{(d_n)}}^{(1)}$ is uniformly distributed on $(0,1)$. It, therefore,
suffices to show that the second term is of order $o(d_n^{-1})$ as well.

As in the proof of Lemma \ref{lem:crucial_approximation_of_n_j} we obtain for
$\sigma_n^2:=\mu_n(1-\mu_n)$
\begin{equation*}
  P\left(U_{\neinsn:n,r}- \mu_n\ge \eps_n\right)\le \exp\left(- \frac{\frac n{\sigma_n^2} \eps_n^2}{3\left(1+\frac{\eps_n}{\sigma_n^2}\right)}\right)
\end{equation*}
and $\eps_n/\sigma_n^2 = o(1)$. We have, moreover,
\begin{align*}
  \frac{n}{\sigma_n^2}\eps_n^2 &= \delta_n^2 n \frac{c_n}{m_n} \frac{n+1}{\neinsn} \frac{n+1}{n+1-\neinsn}\\
  &\ge \delta_n^2  \frac n{m_n} \frac{n+1}{\neinsn} \frac {nc_n}{1+\frac{nc_n}{j}} \\
  &\ge \frac14 \delta_n^2 \frac n{m_n}
\end{align*}
for large $n$ and, since $\delta_n^2 n/(m_n\log(m_n))\to\infty$ as
$n\to\infty$,
\begin{align*}
  &\frac{m_n^{1/2}}{c_n^{1/2}}\sum_{r=1}^{d_n} P\left(U_{\neinsn:n,r}- \mu_n\ge \eps_n\right) \\
  &\le \frac {d_n}{(m_n c_n)^{1/2}} \exp\left(-\frac1{16} \delta_n^2\frac n{m_n} + \log(m_n)\right)\\
  &= \frac {d_n}{(m_n c_n)^{1/2}} \exp\left(- \frac n{m_n}\delta_n^2 \left(\frac1{16} - \frac{m_n\log(m_n)}{n\delta_n^2}\right)\right)\\
  &=o(1).
\end{align*}

By repeating the above arguments one shows that
$(m_nc_n)^{-1/2}T_n=o_P(1)$ as well, which completes the proof of
Lemma \ref{newlem:crucial_approximation_of_n_j}.
\end{proof}

Now we consider the modified test statistic
\begin{equation*}
  \hat T_n(c):=\frac{\sum_{j=1}^k\left(j\,\hat n_{j,M(n)}(c) - \frac 1k\sum_{\ell=1}^k \ell\,\hat n_{\ell,M(n)}(c)\right)^2}{\frac 1k\sum_{\ell=1}^k \ell\,\hat n_{\ell,M(n)}(c)}
\end{equation*}
which does not depend on the marginal df $F_t$, $t\in[0,1]$, of the process
$\bfX$. The following result is a consequence of the arguments in the proof
of Theorem \ref{thm:limit_distribution_of_T_n} and Lemma
\ref{newlem:crucial_approximation_of_n_j}.

\begin{theorem}\label{newthm:limit_distribution_of_test_statistic_general}
Suppose that the process $\bfX=(X_t)_{t\in[0,1]}\in C[0,1]$ has continuous
marginal df $F_t$, $t\in [0,1]$ and that the pertaining copula process
$\bfU=(F_t(X_t))_{t\in[0,1]}$ satisfies
\eqref{eqn:uniform_delta-neighborhood}. Choose a grid of points
$0=t_1^{(d_n)}<t_2^{(d_n)}<\dots<t_d^{(d_n)}=1$ with $d_n\to\infty$ and
$\max_{1\le i\le d_n-1}\abs{t_{i+1}^{(d_n)} - t_i^{(d_n)}}\to 0$ as
$n\to\infty$. Let $m_n=\abs{M(n)}$ satisfy $m_n\to_{n\to\infty}\infty$ and
let $c=c_n$ satisfy $c_n\to 0$, $m_nc_n\to\infty$, $m_nc_n^{1+\delta}\to0$
and $m_n^2\log(m_n)c_n/n\to 0$ as $n\to\infty$. Then we obtain
\[
\hat T_n(c_n)\to_D \sum_{i=1}^{k-1} \lambda_i\xi_i^2
\]
with $\xi_i$ and $\lambda_i$ as in Theorem
\ref{thm:limit_distribution_of_T_n}.
\end{theorem}

\section{Simulations} \label{sec:simulations}
In this section we provide some simulations, which indicate the performance
of the test statistic $T_n(c_n)$ from Theorem
\ref{thm:limit_distribution_of_T_n}. All computations were performed using
the \textsf{R} package \textsf{CompQuadForm} written by Pierre Lafaye de
Micheaux and Pierre Duchesne. We chose \citeauthor{imhof61}'s
\citeyearpar{imhof61} method for computing the $p$-values of our test
statistics; cf. \citet{duchdemi10} for an overview of simulation techniques
of quadratic forms in normal variables.

Therefore we chose $n=10\,000$, $c_n=c>0$ and $k=2$. We generated $1000$
independent realizations of $T_n(c_n)$ that we denote by $T_n^{(i)}(c_n)$ and
we computed the asymptotic $p$-values
$p_i:=1-F_k\left(T_n^{(i)}(c_n)\right)$, $i=1,\dots,1000$, where $F_k$ is the
df of $\sum_{i=1}^{k-1} \lambda_i\xi_i^2$ in Theorem
\ref{thm:limit_distribution_of_T_n}.

The values $p_i$ are filed in increasing order, $p_{(1)}\le \dots\le
p_{(1000)}$, and we plot the points
\begin{equation*}
\left(\frac j{1001},p_{(j)}\right),\qquad 1\le j\le 1000.
\end{equation*}
This quantile plot is a discrete approximation of the quantile function of
the $p$-value of $T_n(c_n)$, which visualizes the performance of the test
statistic $T_n(c_n)$.

If the underlying copula is in a $\delta$-neighborhood of a GPC, then the
points $(j/1001, p_{(j)})$, $j=1,\dots,1000$, should approximately lie along
the line $(x,x)$, $x\in[0,1]$, whereas otherwise $p_{(j)}$ should be
significantly smaller than $5\%$ for many $j$.

\begin{figure}[h]
\begin{multicols}{2}[GPC]
  \includegraphics[width=6.2 cm, trim=30 29 30 55, clip]{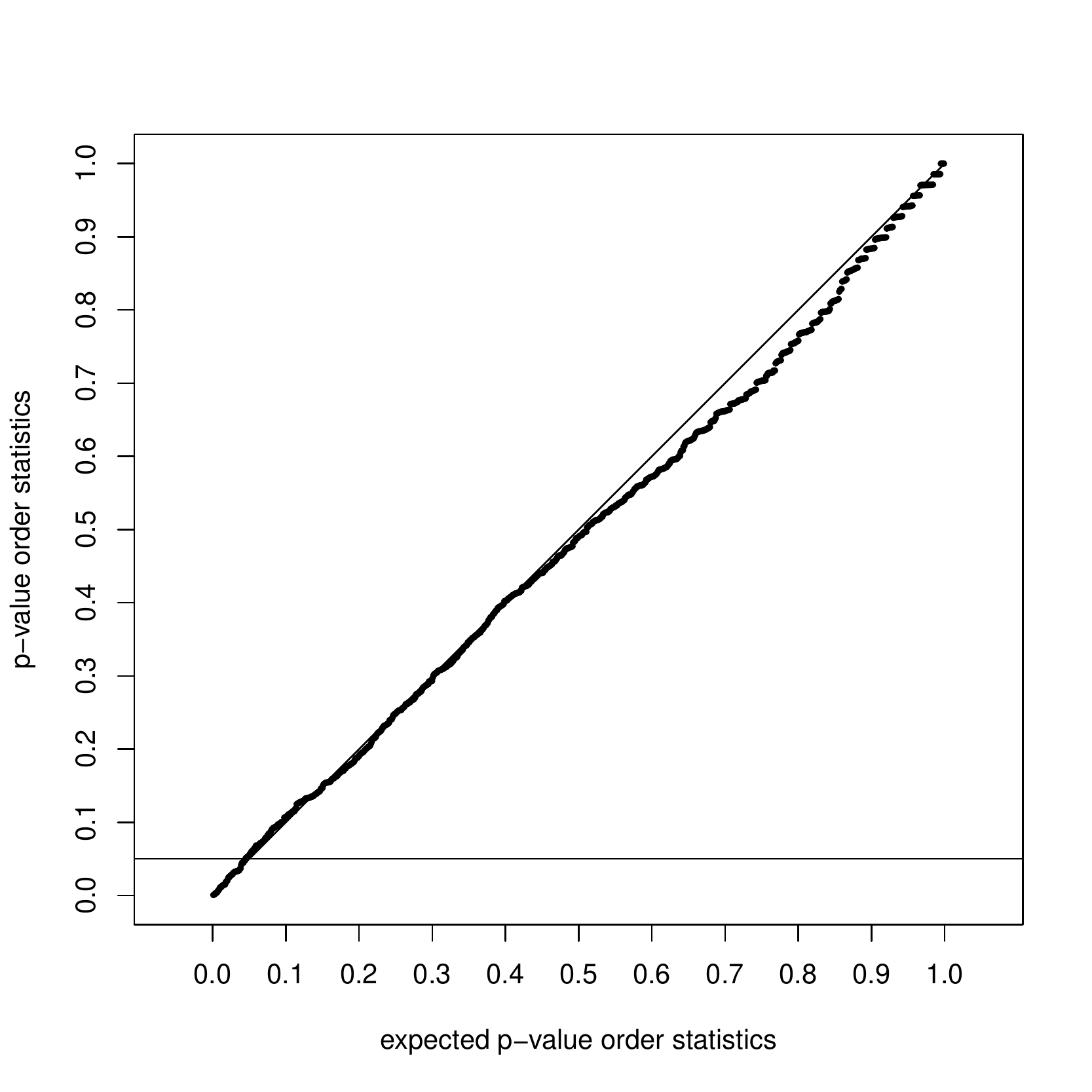}
  \caption{$c=0.2$.}

  \includegraphics[width=6.2 cm, trim=30 29 30 55, clip]{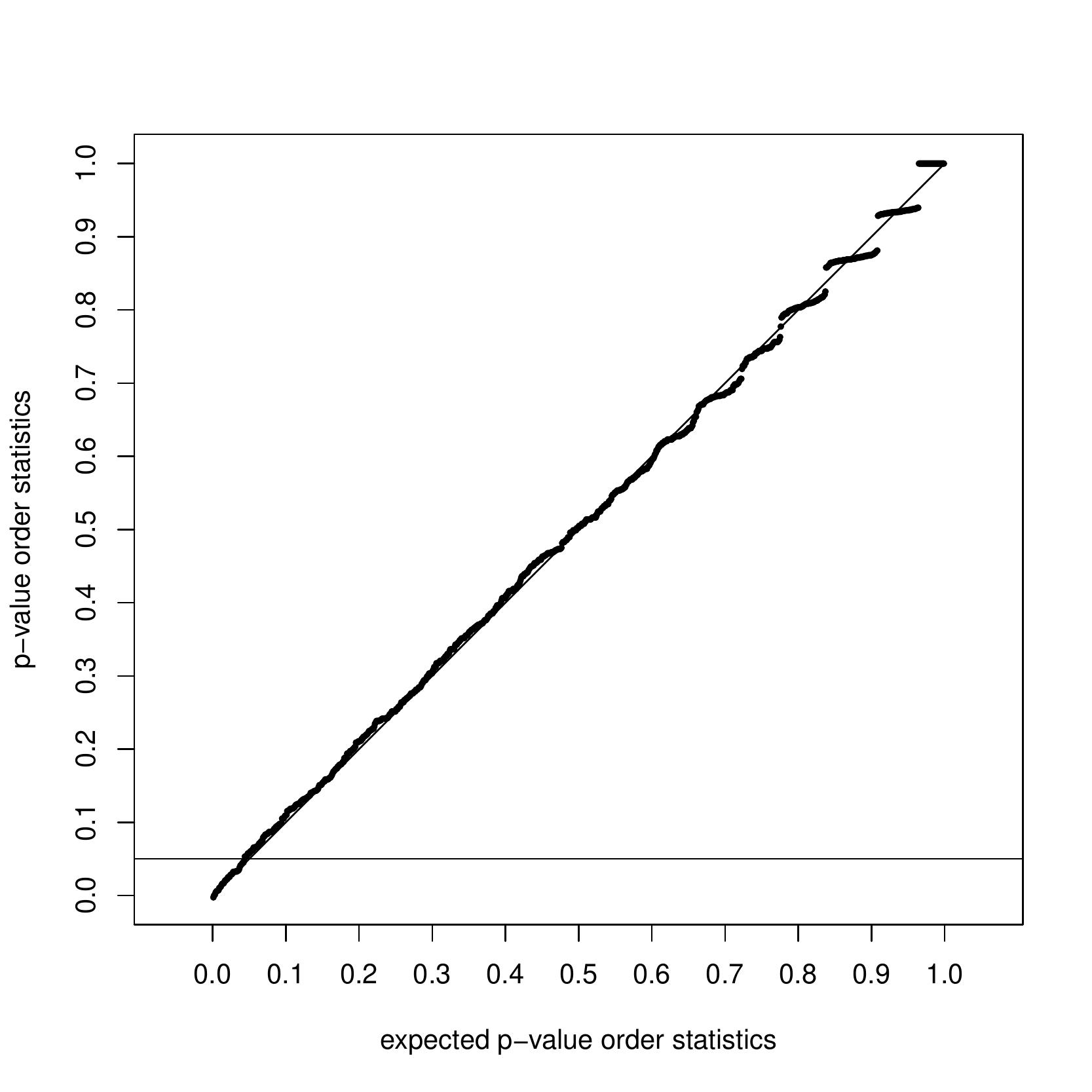}
  \caption{$c=0.01$.}
\end{multicols}
\end{figure}

\begin{figure}[h]
\begin{multicols}{2}[Copula not in $\mathcal D(G)$]
  \includegraphics[width=6.2 cm, trim=30 30 30 55, clip]{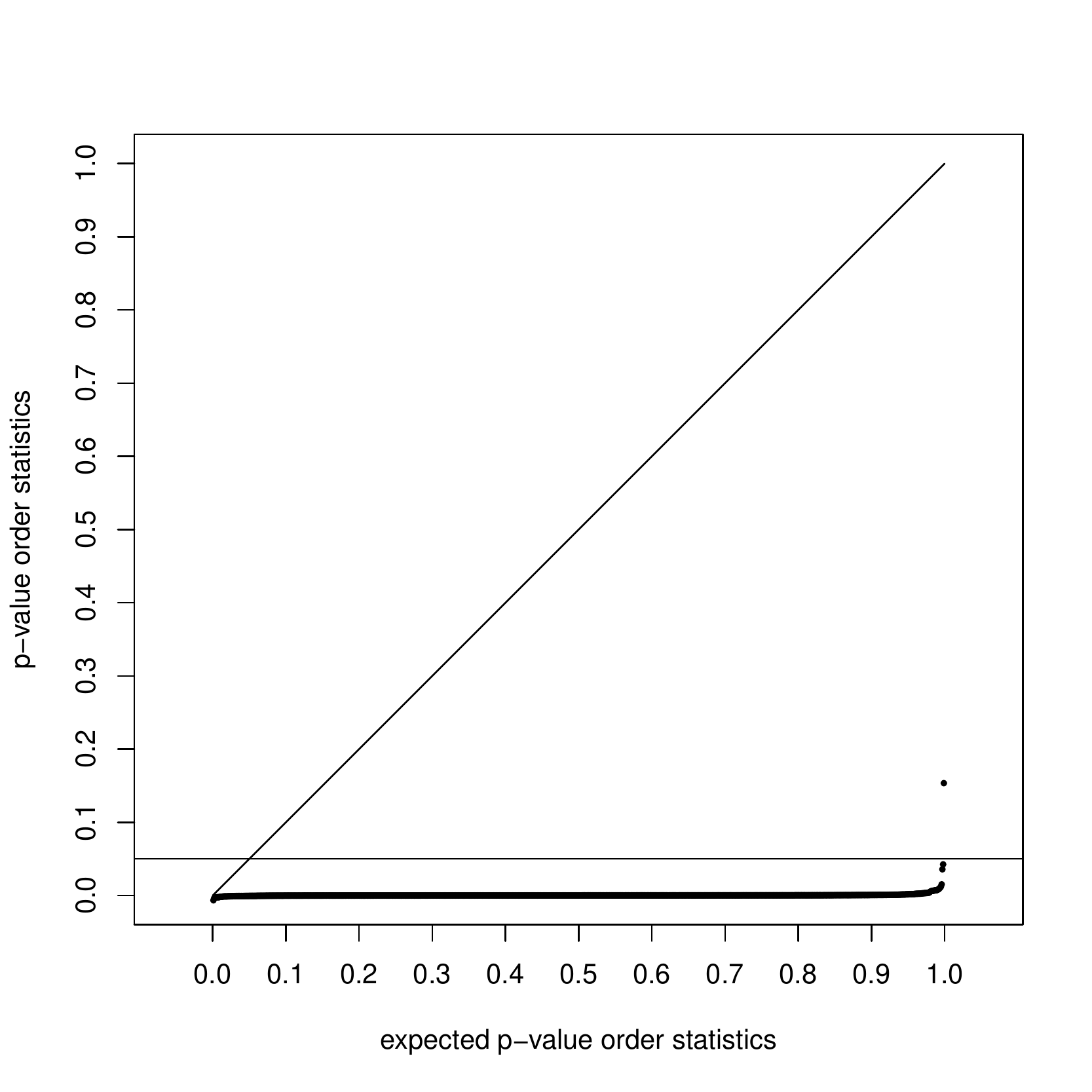}
  \caption{$c=0.2$.} \label{fig:copula_not_in_doa_0_2}

  \includegraphics[width=6.2 cm, trim=30 30 30 55, clip]{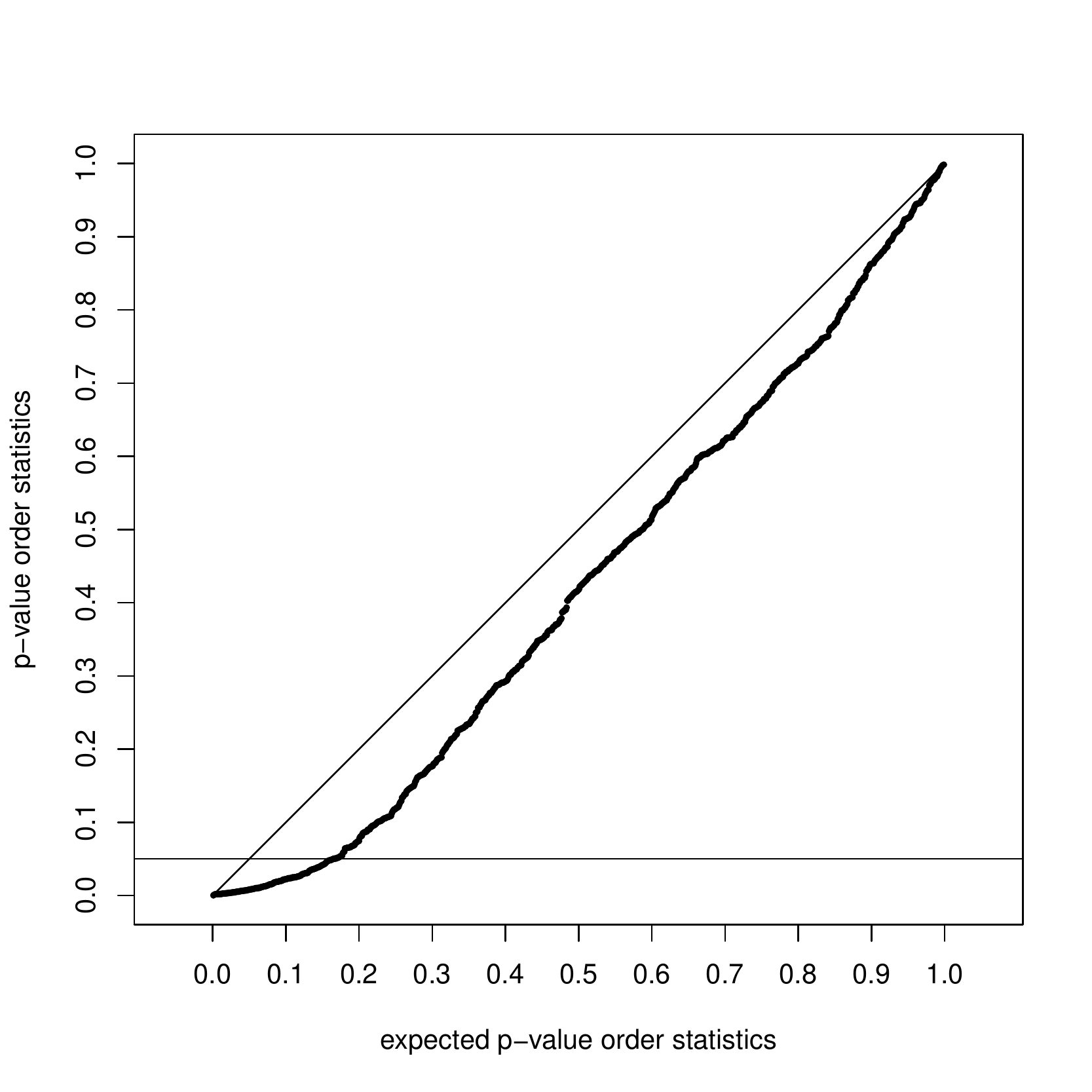}
  \caption{$c=0.01$.} \label{fig:copula_not_in_doa_0_01}
\end{multicols}
\end{figure}

\begin{figure}[h]
\begin{multicols}{2}[Normal Copula with coefficient of correlation $-0.5$]
  \includegraphics[width=6.2 cm, trim=30 29 30 55, clip]{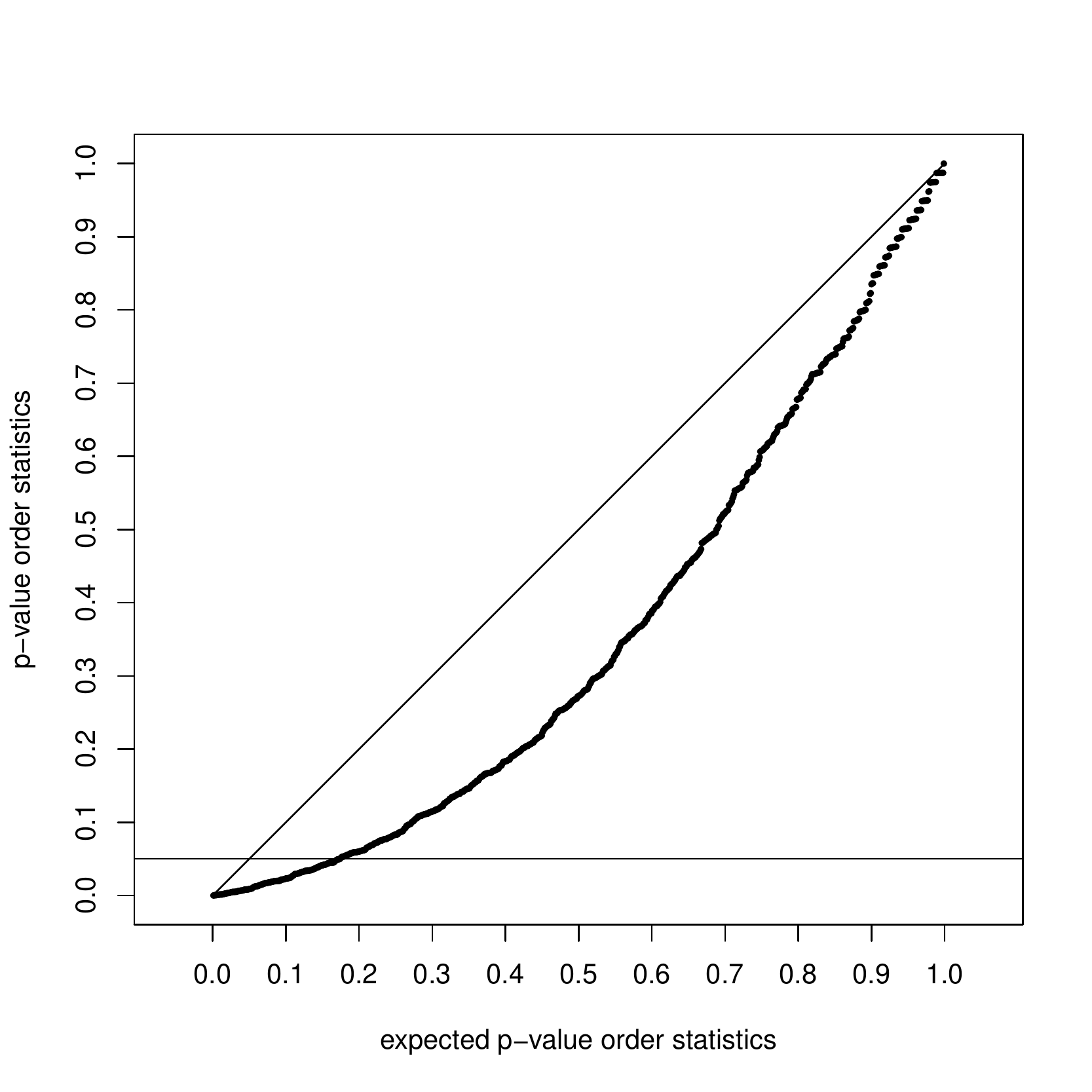}
  \caption{$c=0.2$.} \label{fig:normal_copula_0_2}

  \includegraphics[width=6.2 cm, trim=30 29 30 55, clip]{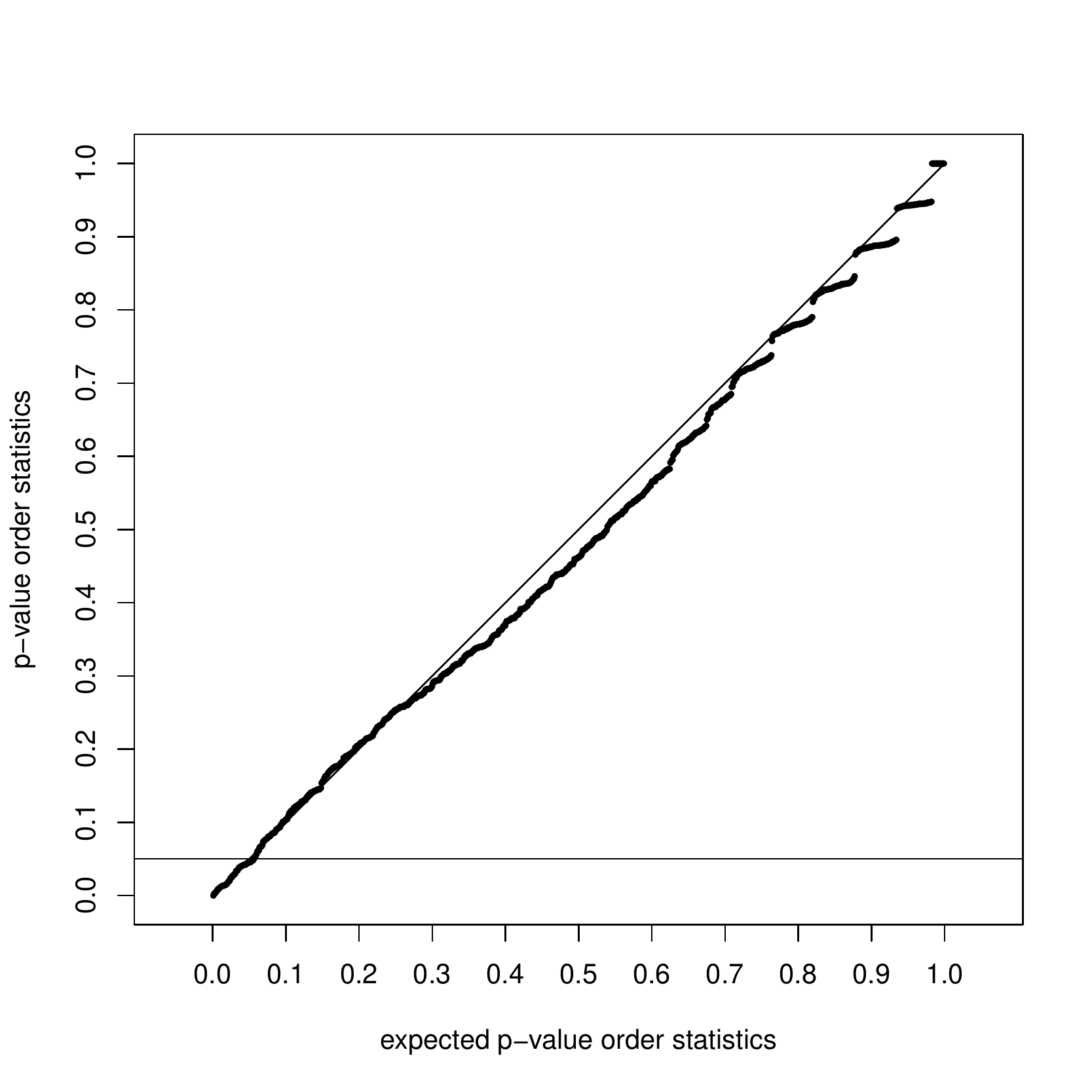}
  \caption{$c=0.01$.}
\end{multicols}
\end{figure}

\begin{figure}[h]
\begin{multicols}{2}[Clayton Copula with parameter $0.5$]
  \includegraphics[width=6.2 cm, trim=30 29 30 55, clip]{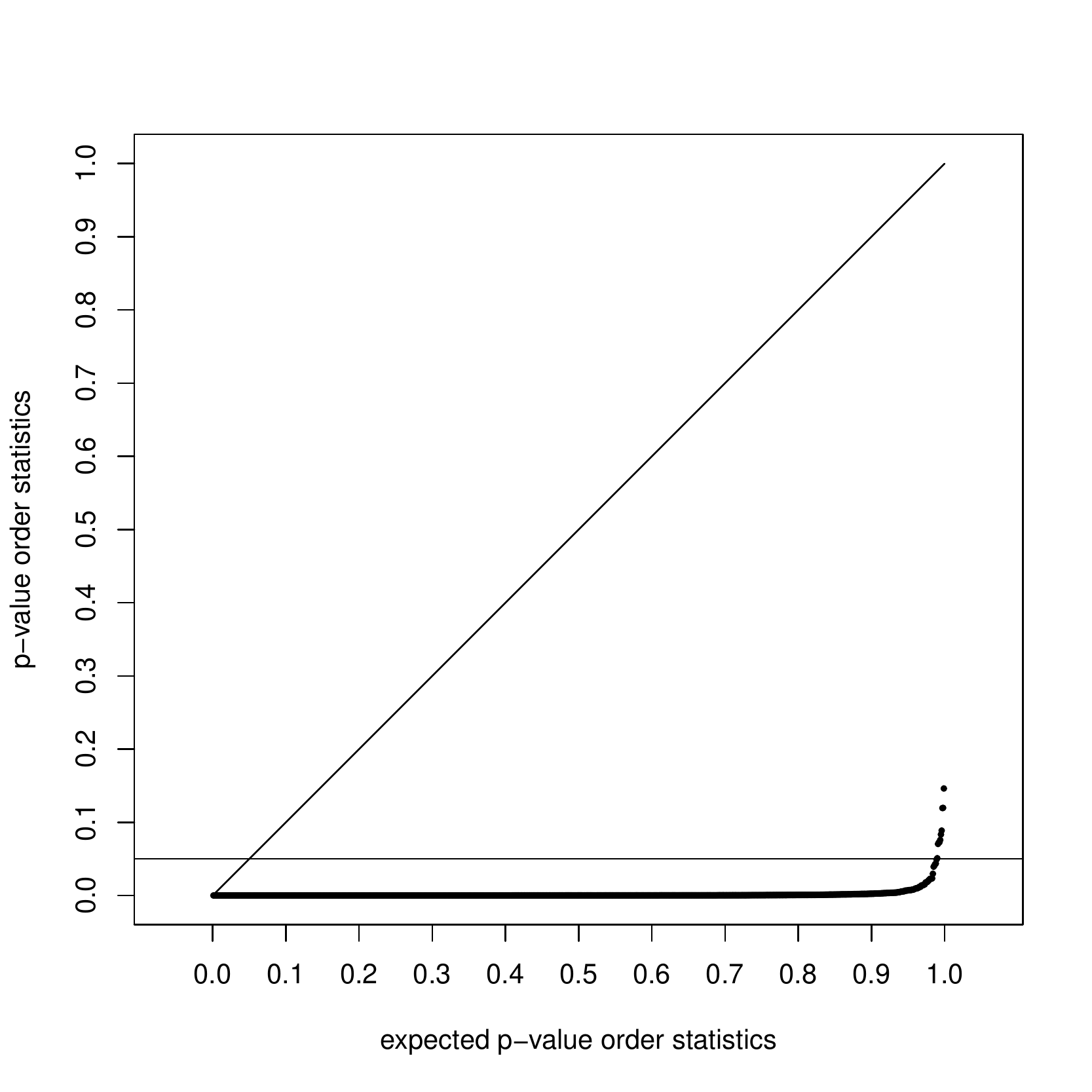}
  \caption{$c=0.2$.} \label{fig:Clayton_copula_0_2}

  \includegraphics[width=6.2 cm, trim=30 29 30 55, clip]{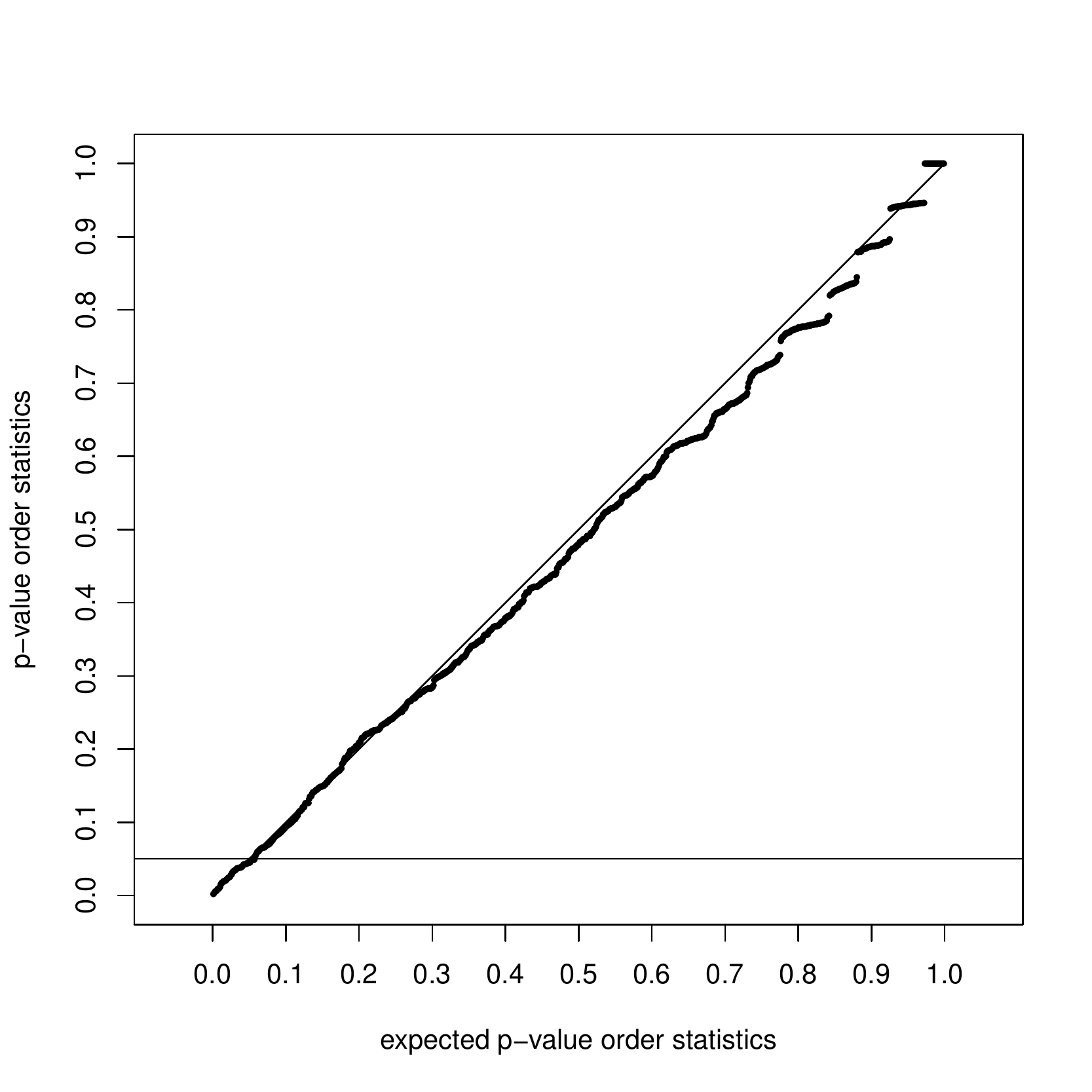}
  \caption{$c=0.01$.}
\end{multicols}
\end{figure}

As can be seen in figures 1--8, the test is quite reliable in detecting a GPC
itself. However, if the underlying copula is \emph{not} a GPC, the
corresponding $p$-value is quite sensitive to the selection of $c$. E.g., if
we decrease the value of $c$ from $0.2$ to $0.01$, a copula that is not even
in the domain of attraction of an extreme value distribution cannot be
detected anymore, cf. Figure \ref{fig:copula_not_in_doa_0_2} and Figure
\ref{fig:copula_not_in_doa_0_01}. On the other hand, there are copulas
satisfying the $\delta$-neighborhood condition that perform well with
$c=0.2$, such as the normal copula in Figure \ref{fig:normal_copula_0_2}, and
those that do not, such as the Clayton copula in Figure
\ref{fig:Clayton_copula_0_2}.

The aforementioned disadvantages can, however, be overcome by considering the
$p$-value as a function of the threshold $c$. Therefore we simulated a single
data set of sample size $n=10\,000$ and plotted the $p$-value for each $c$ of
a some grid $0<c_1<\dots<c_q<1$, see figures 9--12. It turns out that the
$p$-value curve of the considered GPC is above the $5\%$-line for
$c\in(0,0.6)$, roughly. In contrast, the copula in Figure
\ref{fig:copula_not_in_doa} has a peak for intermediate values of $c$ but,
for shrinking $c$, decreases again below the $5\%$-line. Finally, copulas in
the $\delta$-neighborhood of a GPC behave similar to the GPC in Figure
\ref{fig:GPC} except that the point of intersection with the $5\%$-line is
notably smaller than $0.6$.

\begin{figure}
\begin{multicols}{2}
  \includegraphics[width=6.2 cm, trim=30 29 30 55, clip]{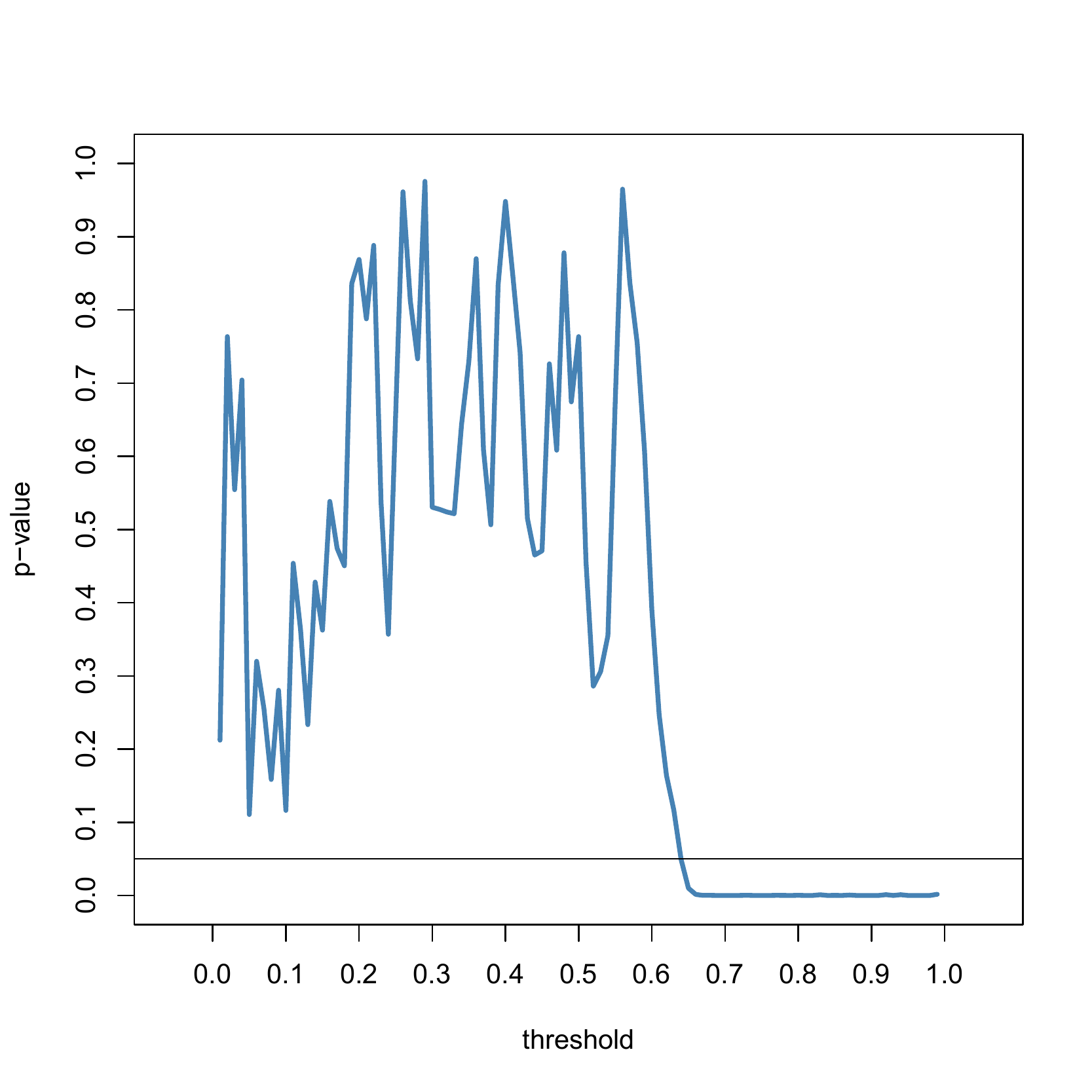}
  \caption{GPC} \label{fig:GPC}

  \includegraphics[width=6.2 cm, trim=30 29 30 55, clip]{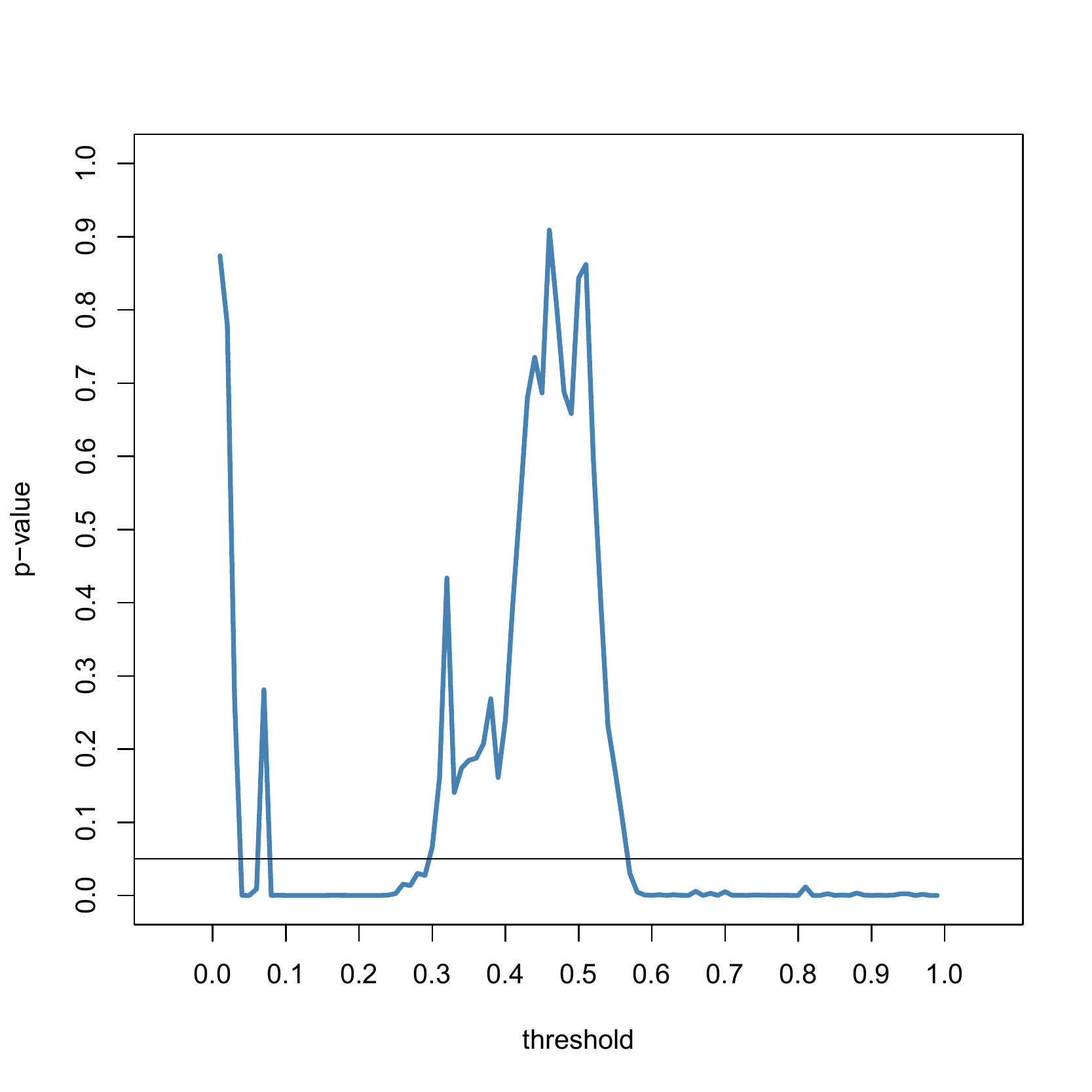}
  \caption{Copula $\notin\mathcal D(G)$} \label{fig:copula_not_in_doa}
\end{multicols}
\end{figure}

\begin{figure}
\begin{multicols}{2}
  \includegraphics[width=6.2 cm, trim=30 29 30 55, clip]{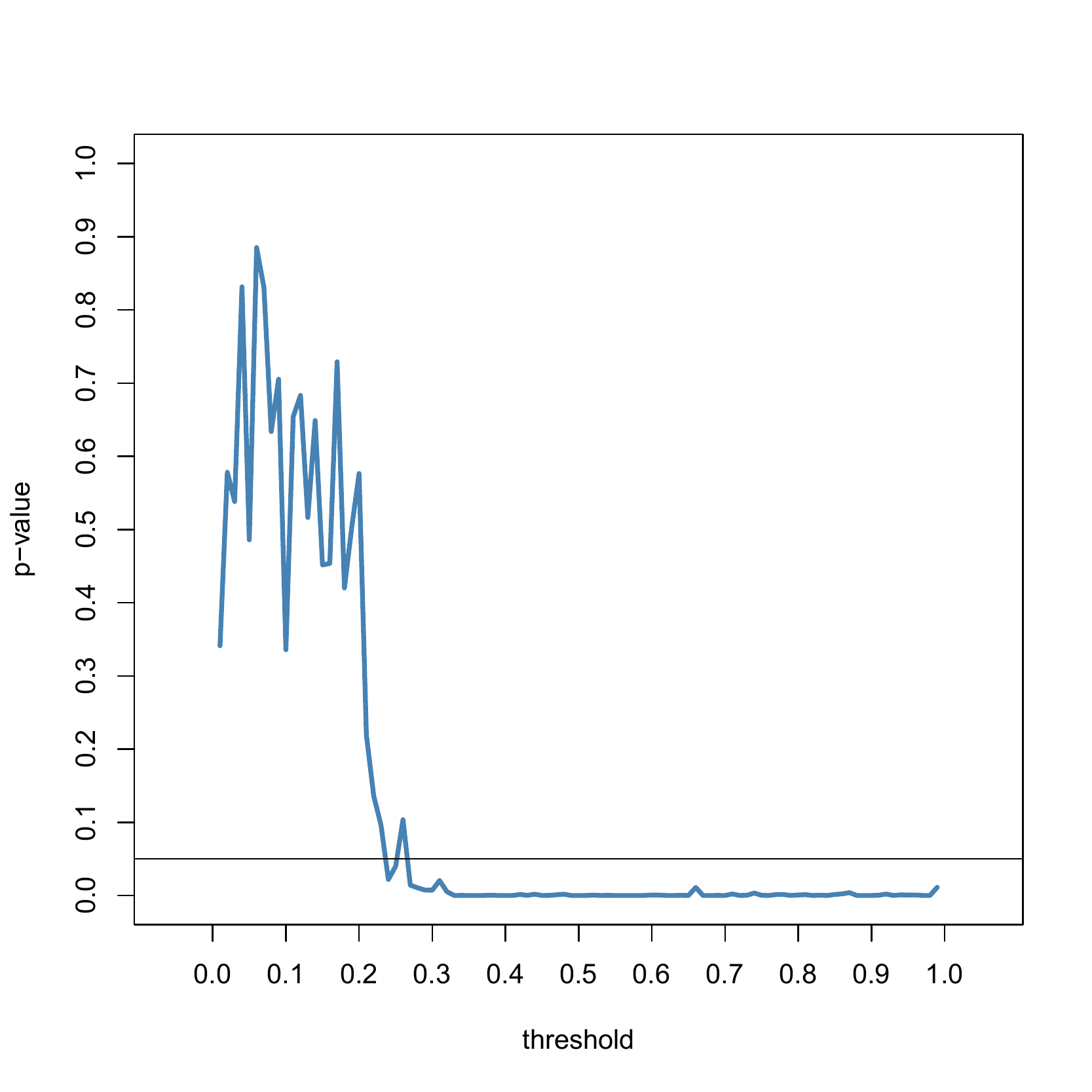}
  \caption{Normal copula}

  \includegraphics[width=6.2 cm, trim=30 29 30 55, clip]{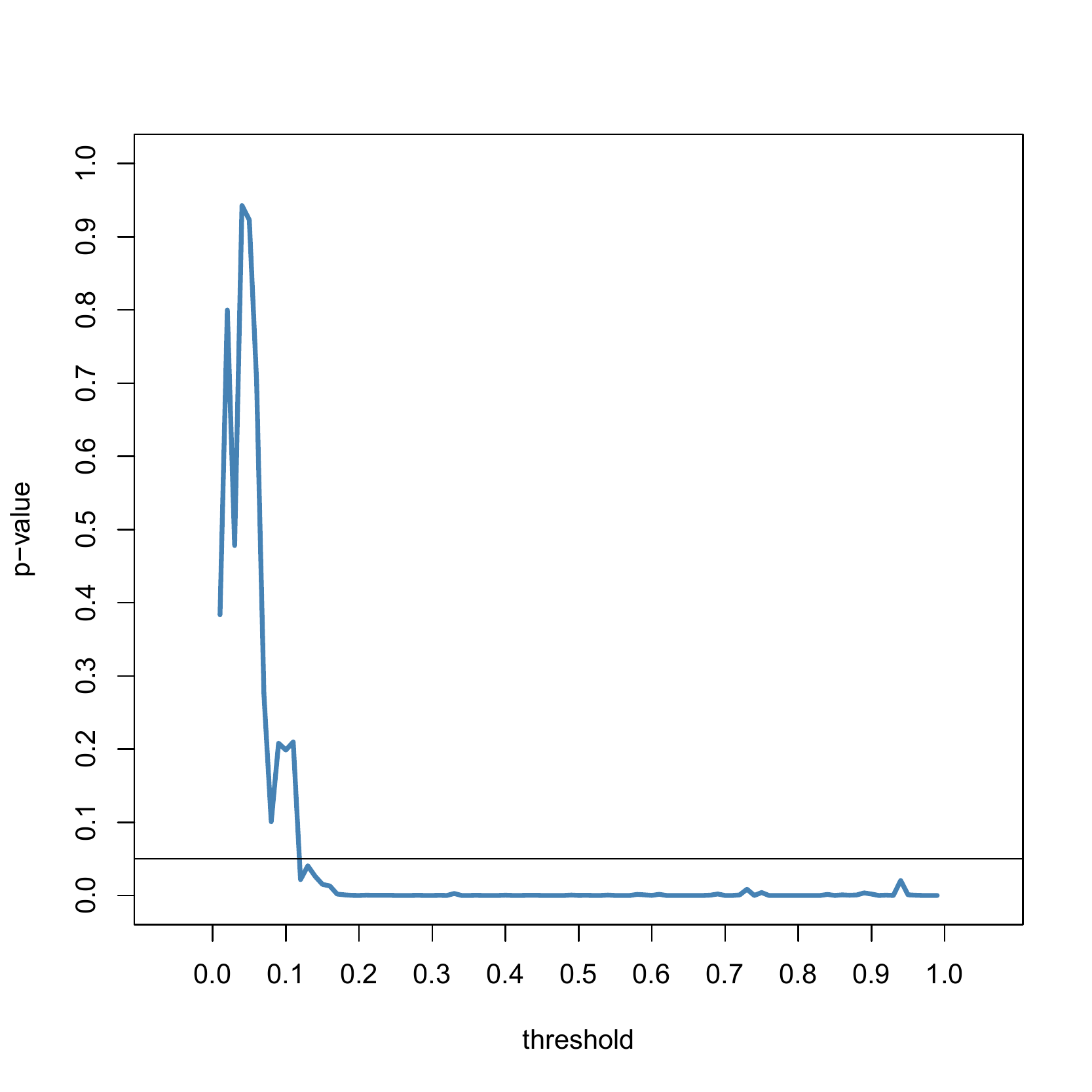}
  \caption{Clayton copula}
\end{multicols}
\end{figure}

The shapes of the $p$-value plots in figures 9--12 seem to be a reliable tool
for the decision whether or not to reject the hypothesis. A great advantage
of this approach is that a practitioner does not need to specify a suitable
value of the threshold $c$, which is a rather complicated task, but can make
the decision based on a highly intelligible graphical tool. A further
analysis of these kind of $p$-value plots is part of future work.

\section*{Acknowledgements}
The authors are grateful to Kilani Ghoudi for his hint to compute the
asymptotic distribution of the above test statistics using
\citeauthor{imhof61}'s \citeyearpar{imhof61} method.

\end{document}